\documentclass[a4paper,twoside,12pt]{amsart}
\usepackage{amsthm, amsmath, amssymb}
\usepackage[margin=4cm]{geometry} 
\usepackage{xy} \xyoption{all}

\usepackage{Laplace}

\begin{document}

\title[On the Laplace transform]{On the Laplace transform for tempered holomorphic functions}

\author[A. D'Agnolo]{Andrea D'Agnolo}

\address{Dipartimento di Matematica\\ 
Universit\`a di Padova\\ 
via Trieste 63\\
35121 Padova, Italy}

\email{dagnolo@math.unipd.it}

\thanks{The author wishes to thank R.I.M.S.~of Kyoto University for
kind hospitality during the preparation of this paper and acknowledges
partial support from the Fondazione Cariparo through the
project ``Differential methods in Arithmetic, Geometry and Algebra''.}



\maketitle

\begin{abstract}
  In order to discuss the Fourier-Sato transform of not necessarily
  conic sheaves, we compensate the lack of homogeneity by adding an
  extra variable. We can then obtain Paley-Wiener type results, using
  a theorem by Kashiwara and Schapira on the Laplace transform for
  tempered holomorphic functions. As a key tool in our approach, we
  introduce the subanalytic sheaf of holomorphic functions with exponential
  growth, which should be of independent interest.
\end{abstract}

\setcounter{tocdepth}{1}
\tableofcontents

\section{Introduction}

Let $\VV$ and $\WW$ be dual $n$-dimensional complex vector spaces.
Kashiwara-Schapira~\cite{KS97} proved that the Laplace transform
\[
\varphi(x) \mapsto \int e^{-\langle x,y \rangle} \varphi(x) dx
\]
induces an isomorphism
\begin{equation}
\label{eq:KS}
\RHom(F,\sho^t_{\VV|\PC(\VV)}) \simeq
\RHom(F^\wedge[n],\sho^t_{\WW|\PC(\WW)}).
\end{equation}
Here, $F$ is a conic $\R$-constructible complex of sheaves on $\VV$,
$F^\wedge$ its Fourier-Sato transform, and $\sho^t_{\VV|\PC(\VV)}$ is
the subanalytic sheaf of holomorphic functions tempered up to the
projective compactification $\PC(\VV)$ of $\VV$.  Let us recall a
couple of statements deduced from \eqref{eq:KS} for particular choices
of $F$.

\smallskip\noindent(i)
Assume that $\VV$ and $\WW$ are complexifications of the real vector spaces $\V$ and $\W$.
For $F=\field_\V$ the constant sheaf on $\V$, one recovers the classical isomorphism
\[
\sect(\V;\Db_{\V|\PR(\V)}^t) \simeq \sect(\W;\Db_{\W|\PR(\W)}^t)
\]
between the spaces of tempered distributions.

\smallskip\noindent(ii)
Let $(x',x'')$ be the coordinates on $\VV =
\C^p \times \C^q$ and $(y',y'')$ the dual coordinates on $\WW$.
Let $A = {\{ (\Re x')^2-(\Re x'')^2 \leq 0 \}}$ be a quadratic cone in $\V$.
For $F=\field_A$, one recovers a result of Faraut-Gindikin:
\[
\sect_A(\V;\Db_{\V|\PR(\V)}^t) \simeq H^q\rsect_{\{(\Re y')^2-(\Re y'')^2 \geq
  0\}}(\WW,\sho_{\WW|\PC(\WW)}^t).
\]

\medskip

Our aim in this paper is to extend the isomorphism \eqref{eq:KS} in
order to treat the case where $F$ is not necessarily conic. This will
allow us to obtain Paley-Wiener type results like the following:

\smallskip\noindent(iii)
Let $A\subset\V$ be a closed, convex, subanalytic, bounded subset.
Denote by $\suppfct_A$ its support function. 
For $F=\field_A$ we will recover the classical Paley-Wiener theorem of
\cite[Theorem~7.3.1]{Hor83}:
\begin{multline*}
\sect_A(\V;\Db_\V) \isoto
\{\psi\in \sect(\WW;\sho_{\WW}) \colon \\
\exists c,\ \exists m,\ \forall y,\ |\psi(y)| \leq c(1+|y|)^m
e^{\suppfct_A(-\Re\,y)} \}.
\end{multline*}
We will also discuss the case where $A$ is not necessarily bounded
nor included in the real part $\V$ of $\VV$.

\smallskip\noindent(iv) Generalizing (ii) above, for $c\geq0$
let $A = {\{ (\Re x')^2-(\Re x'')^2 \leq c^2 \}}$ be a quadric in $\V$.
For $F=\field_A$ we will get a description of
the Laplace transform of the space
$\sect_A(\V;\Db_{\V|\PR(\V)}^t)$.

\medskip

In order to state our result, let us start by describing the functional spaces that
will appear in the statement.

Let $j\colon X \to X'$ be an open subanalytic embedding of real
analytic manifolds.  A $j$-$\R$-constructible sheaf on $X$ is a sheaf
(or more precisely, an object of the derived category of sheaves)
whose proper direct image by $j$ is $\R$-constructible in $X'$.  Such sheaves
are naturally identified with sheaves on the site $X_{j\text-\sa}$,
whose objects are open subsets of $X$ which are subanalytic in $X'$
and whose coverings are locally finite in $X'$.

For $f\colon X \to \R$ a continuous subanalytic function, consider the
sheaf $\cit[f]_{X|X'}$ on $X_{j\text-\sa}$ whose sections on $U\subset
X$ are $f$-tempered functions.  These are smooth functions $\varphi$
which, together with all of their derivatives, locally satisfy on $X'$
an estimate of the type
\[
|\varphi(x)| \leq 
c  \left( 1+ \frac1{\dist(X'\setminus U,x)} + |f(x)| \right)^m e^{f(x)}.
\]
The subanalytic sheaf of tempered functions considered in~\cite{KS01}
is recovered as $\cit_X = \cit[0]_{X|X}$. The sheaf $\cit_{X|X'} =
\cit[0]_{X|X'}$ takes also into account growth conditions at infinity.
We show that $\varphi$ is $f$-tempered on $U$ if and only if
$\varphi(x)e^s$ is tempered on $\{(x,s)\colon x\in U,\ s < -f(x)\}$.

Let now $j\colon X \to X'$ be an open subanalytic embedding of complex
analytic manifolds. We denote by $\sho^{\temp f}_{X|X'}$ the Dolbeault
complex with coefficients in $\cit[f]_{X|X'}$. These sheaves should be
of independent interest in dealing with holonomic
$\D$-modules which are  not necessarily regular.

The functional spaces we
will be dealing with are those of the form
\[
\RHom(F,\sho^{\temp f}_{X|X'}),
\] 
for $F$ a $j$-$\R$-constructible sheaf.  For example, the subanalytic
sheaf of holomorphic functions tempered up to infinity appearing in
\eqref{eq:KS} is recovered as $\sho^t_{X|X'} = \sho^{\temp 0}_{X|X'}$.

\medskip

We can now state our result on the Laplace transform.
Recall that $\VV$ and $\WW$ are dual complex vector spaces of dimension
$n$.
In order to treat the case of not necessarily conic sheaves, we
compensate the lack of homogeneity by adding an extra
variable. Consider the embedding
\[
i\colon \VV\to\VV\times\C,\quad x\mapsto(x,-1).
\]
Let $F$ be a $j$-$\R$-constructible sheaf on $\VV$, where
$j\colon\VV\to\PC(\VV)$ is the complex projective compactification.
Note that if $F$ is conic, then
\[
(\reim i F)^\wedge \simeq (F^\wedge\etens \field_\C)
\tens \field_{\{\Re t \geq 0\}},
\]
with $t\in\C$ the dual of the extra variable.
For $F$ not necessarily conic, assume that 
\[
(\reim i F)^\wedge \simeq (G\etens \field_\C)
\tens \field_{\{\Re t \geq -g(y)\}},
\]
for $G$ a conic sheaf on $\WW$ and $g\colon \WW\to \R$ a continuous
subanalytic function which is positive homogeneous of degree one. (We claim in Conjecture~\ref{con:Tam} that this assumption is not very strong.) Then we have
an isomorphism
\[
\RHom(F,\sho^t_{\VV|\PC(\VV)}) \simeq
\RHom(G[n], \sho^{\temp g}_{\WW|\PC(\WW)}).
\]
In order to get this result we will start by discussing some
generalities on conic sheaves. For example, we explicitly describe
the left and right adjoint to the embedding of conic sheaves into
sheaves. We also prove that the functor $(\reim i
(\cdot))^\wedge$ is fully faithful and that its essential image consists of
the conic sheaves $H$ on $\WW\times\C$ such that
\[
H \conv
\field_{\{\Re t\geq 0,y=0\}} \isoto H.
\] 
This is the kind of condition considered by Tamarkin~\cite{Tam10}.  We
will see how the Fourier transform considered in~\cite{Tam10} is
related to the functor $F\mapsto (\reim i F)^\wedge$.

\medskip

The plan of the paper is as follows.

Section~\ref{se:ker} recalls the formalism of kernel calculus for
sheaves, which will be useful on several occasions in the rest of the
paper.

On a space $X$ endowed with an $\R^+$-action, Section~\ref{se:con} gives an
elementary construction of the left and right adjoint to the embedding of
conic sheaves into sheaves. These are called conification functors.

Let $i\colon Y\to X$ be the embedding of a locally closed subset satisfying a
suitable assumption with respect to the $\R^+$-action.
Section~\ref{se:conified} characterizes the image of the fully faithful
functor sending a not necessarily conic sheaf on $Y$ to the
conification of its direct image by $i$.

Section~\ref{se:FS} recalls some properties of the Fourier-Sato
transform between conic sheaves on dual real vector spaces $\V$ and
$\W$.

Section~\ref{se:iF} characterizes the image of the fully faithful
functor from $\V$ to $\W\times\R$ sending a not necessarily conic
sheaf on $\V$ to the Fourier-Sato transform of its direct image by the
embedding $i\colon\V\to\V\times \R$, $x\mapsto (x,-1)$.

Let $j\colon X\to X'$ be an open subanalytic embedding of real
analytic manifolds. The $j$-subanalytic site on $X$ is the one induced
by the subanalytic site on $X'$. In Section~\ref{se:expR} we introduce
the $j$-subanalytic sheaf of smooth functions with exponential
growth. We also relate this sheaf to the sheaf of tempered smooth
functions with an extra variable.

If $X$ is a complex analytic manifold, we discuss in
Section~\ref{se:expC} the $j$-subanalytic sheaf of holomorphic
functions with exponential growth. This is the Dolbeault complex of
the previous sheaf.

In Section~\ref{se:L} we recall a theorem by Kashiwara and Schapira on the
Fourier-Laplace transform between tempered holomorphic functions associated
with conic sheaves on dual complex vector spaces $\VV$ and $\WW$. 

We then extend it in Section~\ref{se:coL} to sheaves which are not
necessarily conic by considering as above the embedding
$i\colon\VV\to\VV\times \C$, $x\mapsto (x,-1)$.

As an application of the above results, we get in Section~\ref{se:PW}
some Paley-Wiener type theorems.

In Appendix~\ref{se:Tam} we show how the functor $F\mapsto (\reim i
F)^\wedge$ is expressed in terms of the Fourier transform considered
by Tamarkin in~\cite{Tam10}.

\subsection*{Acknowledgments}
The author is grateful to Pierre Schapira for
bringing to his attention the problem of extending the results
of~\cite{KS97} to the not necessarily conic case. He also acknowledges
very useful discussions with St\'ephane Guillermou, Luca Prelli, Pierre
Schapira, and especially with Masaki Kashiwara.

\section{Kernel calculus}\label{se:ker}

We recall here the definition and basic properties of kernel
calculus for sheaves. This formalism will be useful in the first part of the
paper.

\medskip

Let $X$ be a locally compact topological space and $\field$ a field.
For $A\subset X$ a locally closed subset, we denote by $\field_A$ the
constant sheaf on $A$ with stalk $\field$, extended by zero to $X$.
Denote by $\BDC(\field_X)$ the bounded derived category of sheaves of
$\field$-vector spaces on $X$ and by $\tens$, $\rhom$, $\opb f$,
$\roim f$, $\reim f$, $\epb f$ the usual operations (here $f\colon
X\to Y$ is a continuous map of locally compact spaces). More generally,
in this paper we will follow the notations of \cite{KS90}.

Let $X_i$ ($i\in\N$) be locally compact topological spaces. Consider
the projections $q_{ij}\colon X_1\times X_2 \times X_3 \to X_i\times X_j$.
For $K_{ij}\in\BDC(\field_{X_i\times X_j})$, set
\begin{align}
  \label{eq:comp1}
  K_{12}\comp[X_2] K_{23} &= \reim{q_{13}}(\opb{q_{12}}K_{12}\tens\opb{q_{23}}K_{23}),\\
  \label{eq:comp2}
  [K_{12}, K_{23}]\low{X_2} &=
  \roim{q_{13}}\rhom(\opb{q_{12}}K_{12},\epb{q_{23}}K_{23}).
\end{align}
By adjunction and projection formula, one gets
\begin{equation}
  \label{eq:compadj}
[K_{12}\comp[X_2] K_{23},K_{34}]\low{X_3} \simeq
[K_{12},[K_{23},K_{34}]\low{X_3}]\low{X_2}.
\end{equation}

The operations \eqref{eq:comp1} and \eqref{eq:comp2} are called
compositions of kernels. This is because, for
$K\in\BDC(\field_{X\times Y})$, the functors
\begin{align}
  \label{eq:compK1}
  &\BDC(\field_X) \to \BDC(\field_Y), \quad F\mapsto F\comp[X] K,\\
  \label{eq:compK2}
  &\BDC(\field_Y) \to \BDC(\field_X), \quad G\mapsto [K,G]\low Y,
\end{align}
can be considered as sheaf theoretical analogues of an integral transform with kernel
$K$. Note that \eqref{eq:compadj} implies that \eqref{eq:compK1} and \eqref{eq:compK2} are
adjoint functors.

Denote by $\Gamma_f\subset X\times Y$ the graph of $f\colon
X\to Y$ and by $\etens$ the exterior tensor product. One has
\[
  F\etens G \simeq F\comp[\{pt\}]G, \quad \reim f F \simeq F\comp[X]
  \field_{\Gamma_f}, \quad \opb f G \simeq \field_{\Gamma_f} \comp[Y],
  G.
\]
Similar relations
hold for the adjoint operations.

Denote by $\Delta_{X_i}$ the diagonal of $X_i\times X_i$ and set $K_{ij}^r
= \reim r K_{ij}$, for $r(x_i,x_j) = (x_j,x_i)$. One has
\begin{gather}
  \field_{\Delta_{X_1}} \comp[X_1] K_{12} \simeq K_{12} \simeq
  [\field_{\Delta_{X_1}}, K_{12}]\low{X_1}, \\
  (K_{12}\comp[X_2] K_{23})\comp[X_3] K_{34} \simeq K_{12}\comp[X_2]
  (K_{23}\comp[X_3] K_{34}),\\
  \label{eq:compr}
  (K_{12}\comp[X_2] K_{23})^r \simeq K_{23}^r\comp[X_2] K_{12}^r.
\end{gather}

Consider the projections
$q_{ij}\colon X_1\times \cdots\times X_{m+1} \to X_i\times X_j$. One has
\begin{equation}
  \label{eq:comp}
  K_{12}\comp[X_2] \cdots \comp[X_m] K_{mm+1} \simeq
  \reim{q_{1m+1}}(\opb{q_{12}}K_{12}\tens\cdots\tens\opb{q_{mm+1}}K_{mm+1}).
\end{equation}

\section{Conic sheaves}\label{se:con}

Let $X$ be a locally compact space endowed with an action of the
multiplicative group $\R^+$ of positive real numbers and
consider the maps
\[
\xymatrix{ \R^+\times X \ar@<.5ex>[r]^-p \ar@<-.5ex>[r]_-\mu & X, }
\]
where $p$ is the projection and $\mu$ is the action.  
We will write for short $\mu(t,x) = tx$.
Recall (see~\cite[\S3.7]{KS90}) that a sheaf $F$ on $X$ is called $\R^+$-conic if it
satisfies
\[
\opb\mu F \simeq \opb p F.
\]
Note that $\epb p \simeq \opb p [1]$ and $\epb \mu \simeq \opb \mu [1]$.
Denote by
$\BDC_{\R^+}(\field_X)$ the full triangulated subcategory of
$\BDC(\field_X)$ whose objects are $\R^+$-conic.

\begin{definition}
  The left and right conification functors are the pair of
  adjoint functors
  \begin{align*}
    (\cdot)^\cone &\colon \BDC(\field_X) \to \BDC(\field_X), &
    F^\cone &= \reim\mu \opb p F [1] \simeq \reim\mu \epb p F,\\
    {}^\cone(\cdot) &\colon \BDC(\field_X) \to \BDC(\field_X), &
    {}^\cone F
    &=  \roim p \epb\mu F [-1] \simeq \roim p \opb\mu F.
  \end{align*}
\end{definition}

Note that $\mu$ and $p$ can be interchanged in the above
definition. In fact, one has for example $\reim \mu \opb p \simeq
\reim\mu\reim e \opb e \opb p \simeq \reim p \opb \mu$, where $e(t,x)
= (t^{-1},tx)$. 

\medskip

In this section we will show that the left and
right conification functors are respectively the left and
right adjoint to the embedding $\BDC_{\R^+}(\field_X)\to\BDC(\field_X)$.

\begin{remark}
Let $j\colon X\to\R^+\times X$ be the embedding $x\mapsto(1,x)$.
Assume that there is an isomorphism $\opb\mu F \simeq \opb p F$.
Since $\opb p$ is fully faithful, one checks that there is a unique
isomorphism $\beta\colon\opb\mu F \isoto \opb p F$ such that $\opb j\beta=\id_F$.
Thus, the category $\BDC_{\R^+}(\field_X)$ is equivalent to the
equivariant derived category in the sense of \cite{BL94}. There, for
groups $G$ more general than $\R^+$, it is shown that the
forgetful functor $\BDC_G(\field_X)\to\BDC(\field_X)$ has a left and a
right adjoint. Here, for $G=\R^+$, we describe these adjoints without
the machinery of equivariant derived categories.
\end{remark}

Consider the projection $q_{23}\colon \R^+\times X\times X \to X\times
X$. Denoting by $(t,x,x')$ a point in $\R^+\times X\times X$, we will
sometimes write $\{x'=tx\}$ instead of $\Gamma_\mu$.

\begin{definition}
The conification kernel is the object of $\BDC(\field_{X\times X})$
given by
\[
  C_X = \reim{q_{23}}\field_{\{x'=tx\}}[1].
\]
\end{definition}

\begin{lemma}
\label{lem:CX}
There are isomorphisms:
\begin{align}
\label{eq:CXi}
&C_X \simeq (C_X)^r,\\
\label{eq:CXii}
&C_X \simeq \field_{r(\Gamma_\mu)}
  \comp[\R^+\times X] \field_{\Gamma_p}[1], & &C_X \simeq \field_{r(\Gamma_p)}
  \comp[\R^+\times X] \field_{\Gamma_\mu}[1], \\
\label{eq:CXiii}
  &C_X \comp[X]\field_{r(\Gamma_\mu)} \simeq C_X
  \comp[X]\field_{r(\Gamma_p)}, & &\field_{\Gamma_\mu}\comp[X] C_X
  \simeq \field_{\Gamma_p}\comp[X] C_X.
\end{align}
\end{lemma}

\begin{proof}
Set $e(t,x,x') = (t^{-1},x',x)$. Then $q_{23}e = r q_{23}$ and
  $e(\{x'=tx\}) = \{x'=tx\}$.  Hence
  \[
  \reim{q_{23}}\field_{\{x'=tx\}} \simeq \reim{q_{23}} \reim e
  \field_{\{x'=tx\}} \simeq \reim r \reim{q_{23}}\field_{\{x'=tx\}}.
  \]
This proves \eqref{eq:CXi}.

Consider the projection
  $q\colon X\times (\R^+ \times X) \times X \to \R^+\times X\times X$,
  $(x,t,\widetilde x,x')\mapsto(t,x,x')$ and set $q'_{14} = q_{23}q$. One has
  \begin{multline*}
  \field_{r(\Gamma_\mu)} \comp[\R^+\times
  X] \field_{\Gamma_p} = \reim{q'_{14}}\field_{\{x=t\widetilde x,\
    x'=\widetilde x\}} \\
\simeq
  \reim{q_{23}}\reim q\field_{\{x=t\widetilde x,\ x'=\widetilde x\}} \simeq
  \reim{q_{23}}\field_{\{x'=tx\}}.
  \end{multline*}
This proves the first isomorphism in \eqref{eq:CXii}. The second one follows using
\eqref{eq:compr} and \eqref{eq:CXi}.

Consider the projection
  $q\colon \R^+ \times X\times X \times X\times\R^+ \to X\times
  X\times\R^+$, $(t,x,x',x'',t')\mapsto (x,x'',t')$.  Consider the
  subsets of the source space
\begin{align*}
M & =\{(t,x,x',x'',t')\colon x'=tx,\,x'=t'x''\}, \\
P &=\{(t,x,x',x'',t')\colon x'=tx,\,x'=x''\}.  
\end{align*}
By \eqref{eq:comp} one has
  \[
  C_X\comp[X]\field_{r(\Gamma_\mu)} \simeq \reim q\field_{M}[1], \qquad
  C_X\comp[X]\field_{r(\Gamma_p)} \simeq \reim q\field_{P}[1].
  \]
  Let $e(t,x,x',x'',t')=(tt',x,t'x',x'',t')$.  Since $q = q\,e$ and
  $e(P) = M$, one has
  \[
  \reim q \field_{P} \simeq \reim q \reim e \field_{P} \simeq \reim q
  \field_{e(P)} \simeq \reim q \field_{M}.
  \]
This proves the first isomorphism in \eqref{eq:CXiii}. The second one follows using
\eqref{eq:compr} and \eqref{eq:CXi}.
\end{proof}

Note that for $F\in\BDC(\field_X)$ there is a natural morphism 
\begin{equation}
\label{eq:alpha}
\alpha\colon F \to F^\cone,
\end{equation}
defined as follows.
Consider the embedding $j\colon
  X\to\R^+\times X$, $x\mapsto (1,x)$.
  Then $\alpha$ is given by the composite
  \[
  F \simeq \reim\mu\reim j \epb j \epb p F \to \reim\mu\epb p F \simeq
  F^\cone.
  \]
  In terms of kernels, this is induced by the morphism
  \[
\field_{\{t=1,\,x'=tx\}} \isoto
  \rsect_{\{t=1\}}\field_{\{x'=tx\}}[1] \to \field_{\{x'=tx\}}[1],
  \]
noticing that
$\reim{q_{23}}\field_{\{t=1,\ x'=tx\}} \simeq \field_{\Delta_X}$.
  
\begin{proposition}
  \label{pro:con}
  Let $F\in\BDC(\field_X)$.
  \begin{itemize}
  \item[(i)] There are isomorphisms $F^\cone\simeq F\comp[X]C_X \simeq C_X \comp[X] F$.
  \item[(ii)] $F^\cone$ is $\R^+$-conic.
  \item[(iii)] $F$ is $\R^+$-conic if and only if the morphism
    $\alpha$ in \eqref{eq:alpha} is an isomorphism.
  \item[(iv)] The functor $(\cdot)^\cone$ is left adjoint to the
    fully faithful embedding $\BDC_{\R^+}(\field_X) \to
    \BDC(\field_X)$.  In particular, $F^{\cone\cone} \simeq F^\cone$.
  \item [(v)] Similar results hold for the right conification
    ${}^\cone F  \simeq [C_X,F]\low X$.
  \end{itemize}
\end{proposition}

\begin{proof}
(i) The isomorphism $F\comp[X]C_X \simeq C_X
  \comp[X] F$ follows from \eqref{eq:compr} and \eqref{eq:CXi}.
The isomorphism $F^\cone\simeq C_X \comp[X] F$ follows from \eqref{eq:CXii}.

  \medskip\noindent (ii) We have to prove that $\opb \mu (F^\cone)
  \simeq \opb p (F^\cone)$. This is equivalent to
  \[
  F\comp[X] C_X \comp[X]\field_{r(\Gamma_\mu)} \simeq F\comp[X] C_X
  \comp[X]\field_{r(\Gamma_p)},
  \]
  which follows from \eqref{eq:CXiii}.

  \medskip\noindent (iii) If $\alpha$ is an isomorphism, then $F$ is $\R^+$-conic by (ii).
  If $F$ is $\R^+$-conic, then $F^\cone = \reim p\opb \mu F[1] \simeq
  \reim p\opb p F[1] \simeq F$. Here, the last isomorphism follows from the
  isomorphism $\field_{r(\Gamma_p)}\comp[\R^+\times
  X]\field_{\Gamma_p} \simeq \field_{\Delta_X}[-1]$.

  \medskip\noindent (iv) Let $F\in\BDC(\field_X)$ and
  $H\in\BDC_{\R^+}(\field_X)$. By the analogue of (iii) for the right
  conification functor, one has $H\simeq
  {}^\cone H$. Then
  \begin{align*}
    \Hom[\BDC(\field_X)](F,H)
    &\simeq \Hom[\BDC(\field_X)](F,{}^\cone H) \\
    &\simeq \Hom[\BDC(\field_X)](F^\cone,H) \\
    &\simeq \Hom[\BDC_{\R^+}(\field_X)](F^\cone,H) .
  \end{align*}

  \medskip\noindent (v) The similar results for the right
  conification functor can be obtained by adjunction. For example, let
  us show that ${}^\cone F$ is $\R^+$-conic. For any
  $G\in\BDC(\field_{\R^+\times X})$ one has
  \begin{align*}
    \Hom[\BDC(\field_{\R^+\times X})](G,\epb\mu ({}^\cone F))
    &\simeq \Hom[\BDC(\field_X)]((\reim \mu G)^\cone,F) \\
    &\simeq \Hom[\BDC(\field_X)]((\reim p G)^\cone,F) \\
    &\simeq \Hom[\BDC(\field_{\R^+\times X})](G,\epb p ({}^\cone F)) ,
  \end{align*}
  where the second isomorphism follows from \eqref{eq:CXiii}.
Hence $\epb\mu ({}^\cone F) \simeq \epb p
  ({}^\cone F)$.
\end{proof}

\section{Conified sheaves}\label{se:conified}

Let $X$ be a locally compact topological space endowed with an
$\R^+$-action.

\begin{notation}
Let $S\subset X$ be a locally closed subset.
\begin{itemize}
\item[(i)] Denote by $\BDC_{\langle S \rangle}(\field_X)$ the full
  triangulated subcategory of $\BDC(\field_X)$ whose objects $F$
  satisfy $F_{X\setminus S} = 0$.
\item[(ii)] Denote by $\BDC_S(\field_X)$ the full triangulated
  subcategory of $\BDC(\field_X)$ whose objects $F$ satisfy
  $\rsect_{X\setminus S}F = 0$.
\end{itemize}
\end{notation}

\begin{definition}
 Let us say that a subset $Y\subset X$ is $\R^+$-simple if it is locally closed
 and if the multiplication $\mu$ induces a topological isomorphism
 between $\R^+\times Y$ and the set $\R^+Y\subset X$ endowed with
 the induced topology.
\end{definition}

\begin{lemma}
If $Y$ is $\R^+$-simple, then $\R^+Y$ is locally closed in $X$.
\end{lemma}

\begin{proof}
Since $Y$ is locally closed, it is locally compact. 
Then $\R^+\times Y$ is locally
compact, and thus so is $\R^+Y$ for the induced topology. 
It follows that $\R^+Y$ is locally closed in $X$.
\end{proof}

\begin{proposition}
  \label{pro:icone}
Let $i\colon Y\to X$ be the embedding of an $\R^+$-simple subset.
There are equivalences
\begin{align*}
\BDC(\field_Y) &\isoto \BDC_{\R^+,\langle \R^+Y \rangle}(\field_X),
& G &\mapsto (\reim i G)^\cone,\\
\BDC(\field_Y) &\isoto \BDC_{\R^+,\R^+Y}(\field_X),
& G &\mapsto {}^\cone(\roim i G),
\end{align*}
with quasi inverses $\opb i[-1]$ and $\epb i[1]$, respectively.
\end{proposition}

\begin{proof}
(i)
For the first equivalence,
it is enough to prove that for $G\in\BDC(\field_{Y})$ and
$F\in\BDC_{\R^+}(\field_X)$ there are isomorphisms
  \begin{equation}\label{eq:tempisos}
  \opb i((\reim i G)^\cone) \simeq G[1], \qquad
(\reim i \opb i F)^\cone \simeq F_{\R^+Y}[1].
  \end{equation}
In fact, \eqref{eq:tempisos} proves in particular that 
$(\reim i G)^\cone\in\BDC_{\langle \R^+Y \rangle}(\field_{Y})$, since one has
\[
(\reim i G)^\cone \simeq (\reim i \opb i((\reim i G)^\cone))^\cone[-1]
\simeq ((\reim i G)^\cone)_{\R^+Y}.
\]

\medskip\noindent(i-a) 
One has
  \[
  \opb i((\reim i G)^\cone) \simeq G \comp[Y] \field_{\Gamma_i}
  \comp[X] C_X \comp[X] \field_{r(\Gamma_i)}.
  \]
  Consider the projection $q\colon Y\times X\times \R^+ \times X
  \times Y \to Y\times Y$ given by $q(y,x,t,x',y') = (y,y')$ and the
  subset
\[
Q=\{(y,x,t,x',y')\colon x=i(y),\,x'=tx,\,x'=i(y')\}
\]
of the source space.  
  Since $Y$ is $\R^+$-simple, the equality $i(y') = ti(y)$
  implies $t=1$. Hence 
  \[
  \field_{\Gamma_i} \comp[X] C_X \comp[X] \field_{r(\Gamma_i)} \simeq \reim q\field_{Q}[1]
\simeq \field_{\Delta_Y}[1].
  \]
This proves the first isomorphism in \eqref{eq:tempisos}. 

\medskip\noindent(i-b)
For the second isomorphism, since $F\simeq F^\cone$ one has
\[
(\reim i \opb i F)^\cone \simeq ((F^\cone)_Y)^\cone \simeq 
F \comp[X] C_X \comp[X] \field_Y \comp[X] C_X.
\]
Here, $\field_Y$ denotes the sheaf on $X\times X$, extension by zero of the 
constant sheaf on $Y\subset X = \Delta_X \subset X\times X$.

Consider the map $q\colon X\times \R^+ \times X \times \R^+
\times X \to X\times \R^+ \times X$
given by $q(x,t,x',t',x'') = (x,tt',x'')$ and the subset
\[
Q = \{(x,t,x',t',x'')\colon x' = tx,\, x'\in Y,\, x'' = t'x'\}
\]
of the source space.
Since $Y$ is $\R^+$-simple, $q$ induces a topological isomorphism between $Q$
and the subset 
\[
P = \{(x,t'',x'')\colon x,x''\in\R^+Y,\, x''=t''x\}
\]
of the target space.
Hence
\[
C_X \comp[X] \field_Y \comp[X] C_X \simeq \reim{p}\reim{q}\field_Q[2] \simeq \reim{p}\field_P[2] \simeq
\reim{(j\times j)} C_{\R^+Y}[1],
\]
where $j\colon \R^+Y \to X$ is the embedding and $p\colon X\times \R^+ \times X \to X\times X$ the projection. Then
\begin{multline*}
(\reim i \opb i F)^\cone \simeq F \comp[X] \reim{(j\times j)} C_{\R^+Y}[1] \simeq \\
\reim j (F|_{\R^+Y} \comp[\R^+Y] C_{\R^+Y})[1] \simeq 
\reim j (F|_{\R^+Y})[1] \simeq 
F_{\R^+Y}[1],
\end{multline*}
where the third isomorphism is due to the fact that $F|_{\R^+Y}$ is conic.

\medskip\noindent(ii) For the second equivalence in the statement,
it is enough to prove that for $G\in\BDC(\field_{Y})$ and
$F\in\BDC_{\R^+}(\field_X)$ there are isomorphisms
  \begin{equation}
  \epb i({}^\cone(\reim i G)) \simeq G[-1], \qquad
{}^\cone(\roim i \epb i F) \simeq \rsect_{\R^+Y}F[-1].
  \end{equation}
These can be deduced from \eqref{eq:tempisos} by adjunction. For
example, the second isomorphism follows by noticing that for any 
$F'\in \BDC_{\R^+}(\field_X)$ one has
\begin{align*}
\Hom(F',\rsect_{\R^+Y}F[-1]) &\simeq \Hom(F'_{\R^+Y}[1],F) \\
 &\simeq \Hom((\reim i \opb i F')^\cone,F) \\
 &\simeq \Hom(F',\roim i \epb i ({}^\cone F)) \\
 &\simeq \Hom(F^{\prime\cone},\roim i \epb i F) \\
 &\simeq \Hom(F',{}^\cone(\roim i \epb i F)),
\end{align*}
where the fourth equivalence follows from the fact that $F$ and $F'$
are conic.
\end{proof}

Let $Y,Z$ be locally compact spaces endowed with an
$\R^+$-action. Let $X= Y\times Z$ be endowed with the diagonal
$\R^+$-action. Let $z_\circ\in Z$ be such that $\R^+$ acts regularly
on the orbit $\R^+ z_\circ$ (i.e.\ $\mu$ induces an isomorphism $\R^+ \times \{z_\circ\} \simeq \R^+ z_\circ$). Then the embedding
\[
i\colon Y \to X, \qquad y \mapsto (y,z_\circ)
\]
identifies $Y$ with an $\R^+$-simple closed subset of $X$.

\begin{lemma}
  \label{lem:iFhom}
  For $G\in\BDC_{\R^+}(\field_{Y})$ there is
  an isomorphism
  \[
  (\reim i G)^\cone \simeq G\etens \field_{\R^+ z_\circ}[1].
  \]
\end{lemma}

\begin{proof}
  As $G\simeq G^\cone$, it is enough to prove the isomorphism in $\BDC(\field_{Y \times Y \times Z})$
  \[
  C_Y \comp[Y] \field_{\Gamma_i} \comp[X] C_X \simeq C_Y \etens
  \field_{\R^+ z_\circ}[1].
  \]
  Denote by $w=(t,y,y',t',y'',z,y''',z')$ a point of $\R^+ \times Y
  \times Y \times \R^+ \times Y \times Z \times Y \times Z$, and set
  $Q=\{y'=ty,\,y'=y'',\,z=z_\circ,\,y'''=t'y'',\,z'=t'z\}$. By
  \eqref{eq:comp} one has
  \[
  C_Y \comp[Y] \field_{\Gamma_i} \comp[X] C_X \simeq \reim q
  \field_Q[2],
  \]
  for $q(w)=(y,y''',z')$. Set
  $e(w)=(tt',y,t'y',t^{\prime-1},t'y'',z,y''',z')$. Then
  $qe=q$ and $e(Q) =
  \{y'''=ty,\,y'=y''=y''',\,z=z_\circ,\,z_\circ=t'z'\}$. As $\R^+$
  acts regularly on the orbit of $z_\circ$, one has
  \[
  \reim q \field_{Q} \simeq \reim q \field_{e(Q)} \simeq
(\reim{q_{23}}\field_{\{y'''=ty\}}) \etens \field_{\R^+ z_\circ}.
  \]
\end{proof}

\begin{lemma}
  \label{lem:iAF}
  Let $A\subset Y$ be a locally closed subset. Then there is
  an isomorphism
  \[
  (\reim i \field_A)^\cone \simeq \field_{\R^+i(A)}[1].
  \]
\end{lemma}

\section{Fourier-Sato transform}\label{se:FS}

Here, we recall the definition and main properties of the
Fourier-Sato transform, referring to \cite[\S3.7]{KS90} for details.

\medskip

Let $\V$ and $\W$ be dual real vector spaces by
the pairing
\[
\V\times\W\to\R, \quad (x,y) \mapsto \langle x,y\rangle.
\]
They are endowed with a natural $\R^+$-action.

\begin{definition}
The Fourier-Sato transform and its adjoint are the functors
\begin{align*}
  (\cdot)^\wedge&\colon\BDC(\field_\V)\to \BDC_{\R^+}(\field_{\W}),&
  F&\mapsto F \comp[\V] \field_{\{\langle x,y\rangle \leq 0\}}, \\
  (\cdot)^\vee&\colon\BDC(\field_{\W})\to \BDC_{\R^+}(\field_\V),&
  G&\mapsto [\field_{\{\langle x,y\rangle \leq 0\}},G]\low{\W}.
\end{align*}
One uses the same notations when interchanging the roles of $\V$
and $\W$.
\end{definition}

\begin{theorem}[{\cite[Theorem~3.7.9]{KS90}}]
\label{thm:F}
The Fourier-Sato transform induces an equivalence of
    categories
    \[
    (\cdot)^\wedge \colon \BDC_{\R^+}(\field_\V) \isoto \BDC_{\R^+}(\field_{\W})
    \]
with quasi-inverse $(\cdot)^\vee$.
\end{theorem}

Denote by $n$ the dimension of $\V$.
For $F\in\BDC_{\R^+}(\field_\V)$ one has
\begin{equation}
\label{eq:wedgevee}
F^\vee \simeq
    F^{\wedge a}[n],
\end{equation}
where $G^a = \opb a G$ for $a$ the antipodal map $a(y) = -y$.

Let $\V_i$  ($i=1,2$) be a real vector space of dimension $n_i$.
Denote by ${}^t f\colon \V^*_2 \to \V^*_1$ the transpose of a linear map $f \colon \V_1 \to \V_2$.
For $F_i\in\BDC_{\R^+}(\field_{V_i})$, one has
\begin{align}
	\label{eq:Fouetens}
(F_1\etens F_2)^\wedge &\simeq F_1^\wedge \etens F_2^\wedge, \\
	\label{eq:Foueim}
(\reim f F_1)^ \wedge &\simeq \opb{{}^t f}(F_1^\wedge), \\
(\opb f F_2)^\wedge &\simeq \reim{{}^t f}(F_2^\wedge)[n_2-n_1].
\end{align}

Denote by $s\colon \V\times\V\to\V$ the vector sum $s(x_1,x_2) = x_1+x_2$.
The convolution of $F,F'\in\BDC(\field_\V)$ is defined by
\[
F \conv F' = \reim s(F\etens F').
\]
Following~\cite{Tam10} (see also~\cite{GS11}), the right adjoint to
the convolution is given by
\[
\hom^*(F,F') = \roim s\rhom(\opb{q_2}F^a,\epb{q_1}F').
\]

Noticing that the diagonal embedding is the transpose of the vector
sum, for $F,F'\in\BDC_{\R^+}(\field_\V)$ one gets
\begin{equation}
\label{eq:convF}
(F \tens F')^\wedge \simeq F^\wedge \conv F^{\prime\wedge}[n].
\end{equation}
By adjunction one then has
\begin{equation}
\label{eq:dconvF}
\hom(F, F')^\vee \simeq \hom^*(F^\vee, F^{\prime\vee}).
\end{equation}

\medskip

Let us end this section by recalling some computations of Fourier transforms that we shall
use later.

A subset $\gamma\subset \V$ such that $\gamma = \R^+\gamma$ is called a cone.
A cone $\gamma$ is called proper if it contains no lines. 
Note that $\gamma$ is convex if and only if $\gamma+\gamma=\gamma$.
The polar of $\gamma\subset \V$ is the cone
\[
\gamma^\circ = \{y\in\W \colon \langle x,y \rangle \geq 0, \ \forall
x\in\gamma \}.
\]

\begin{lemma}[{\cite[Lemma 3.7.10]{KS90}}]
  \label{lem:Fex}
  \begin{itemize}
  \item[(i)] Let $\gamma\subset \V$ be a proper closed convex cone
    containing the origin. Then
    \[
    \field_\gamma^\wedge \simeq \field_{\Int \gamma^\circ}.
    \]
  \item[(ii)] Let $\gamma\subset \V$ be an open convex cone. Then
    \[
    \field_\gamma^\wedge \simeq \field_{\gamma^{\circ a}}[-n].
    \]
   \end{itemize}
\end{lemma}

Let $\V = \R^p \times \R^q \times \R^r$. Denote
$(x',x'',x''')$ the coordinate system on $\V$ and by 
$(y',y'',y''')$ the dual coordinate system on $\W$.
Set
\[
x^{\prime 2} = x_1^{\prime 2}+\cdots+x_p^{\prime 2}, \quad
x^{\prime\prime 2} = x_1^{\prime\prime 2}+\cdots+x_q^{\prime\prime 2}.
\]

\begin{lemma}[{\cite[Lemma~6.2.1]{KS97}}]\label{lem:qcone}
Let
\[
\gamma = \{x^{\prime 2}-x^{\prime\prime 2} \leq 0,\ x'''=0\}
\]
be a quadratic cone. Then
\[
\field_\gamma^\wedge \simeq
\field_{\{y^{\prime 2}-y^{\prime\prime 2} \geq 0\}}[-q].
\]
\end{lemma}

\begin{proof}
  The transpose of the embedding $i\colon \R^p \times \R^q \to \V$,
  $(x',x'')\mapsto(x',x'',0)$, is the projection $p\colon \W \to \R^p
  \times \R^q$. By \eqref{eq:Foueim}, one has
\[
\field_{\{x^{\prime 2}-x^{\prime\prime 2} \leq 0,\ x'''=0\}}^\wedge \simeq
(\reim i \field_{\{x^{\prime 2}-x^{\prime\prime 2} \leq 0\}})^\wedge \simeq
\opb p (\field_{\{x^{\prime 2}-x^{\prime\prime 2} \leq 0\}}^\wedge).
\]
We thus reduce to the case $r=0$, discussed in \cite[Lemma~6.2.1]{KS97}.
\end{proof}

\section{Conified Fourier-Sato transform}\label{se:iF}

Here, in order to apply the Fourier-Sato transform to
not necessarily conic sheaves, we will
compensate the lack of homogeneity by adding an extra variable.

\medskip

As in the previous section, let $\V$ and $\W$ be dual real vector spaces.

Note that the conification functor on vector spaces
can be expressed in terms of the Fourier-Sato transform:

\begin{lemma}
  For $F\in \BDC(\field_\V)$ one has
  \begin{align*}
  F^\cone &\simeq F^{\wedge\wedge a}[n],&\qquad 
F^\wedge &\simeq F^{\cone\wedge}, \\
  {}^\cone F &\simeq F^{\vee\vee a}[-n],&\qquad 
F^\vee &\simeq ({}^\cone F)^{\vee}.
  \end{align*}
\end{lemma}

\begin{proof}
Since the arguments are similar, we will only discuss the first two isomorphisms.

  For $H\in\BDC_{\R^+}(\field_\V)$ one has $H\simeq
  H^{\vee\wedge}\simeq H^{\vee\vee a}[-n]$.  Hence there are
  isomorphisms
  \begin{align*}
    \Hom[\BDC(\field_\V)](F,H)
    &\simeq \Hom[\BDC(\field_\V)](F,H^{\vee\vee a}[-n]) \\
    &\simeq \Hom[\BDC(\field_\V)](F^{\wedge\wedge a}[n],H) \\
    &\simeq \Hom[\BDC_{\R^+}(\field_\V)](F^{\wedge\wedge a}[n],H) .
  \end{align*}
  By Proposition~\ref{pro:con}~(iv) and by uniqueness of the left
  adjoint, it follows that $F^\cone \simeq
  F^{\wedge\wedge a}[n]$. One then has $F^{\cone\wedge} \simeq
  F^{\wedge\wedge a \wedge}[n] \simeq F^{\wedge\vee\wedge} \simeq
  F^\wedge$.
\end{proof}

Consider the dual vector spaces
\[
\Vt = \V\times\R, \qquad \Wt = \W \times \R
\]
by the pairing $\langle(x,s),(y,t)\rangle = \langle x,y\rangle +
st$.

\begin{notation}
\label{not:tam} 
\begin{itemize}
\item[(i)]
Denote by $\BDC_{*\{t\geq 0\}}(\field_{\Wt})$ the full triangulated subcategory
of $\BDC(\field_{\Wt})$ whose objects $G$ satisfy
$G \conv 
\field_{\{t\geq 0,\,y=0\}} \isoto G$, or equivalently 
$G\conv\field_{\{t> 0,\,y=0\}} = 0$. 
\item[(ii)]
Denote by $\BDC_{\{t\geq 0\}^*}(\field_{\Wt})$ the full triangulated subcategory
of $\BDC(\field_{\Wt})$ whose objects $G$ satisfy
$G \isoto
\hom^*(\field_{\{t\geq 0,\,y=0\}},G)$, or equivalently 
$\hom^*(\field_{\{t> 0,\,y=0\}},G) = 0$. 
\end{itemize}
\end{notation}

Let us identify $\V$ with an $\R^+$-simple subset of $\Vt$ by the embedding
\[
i\colon \V \to \Vt, \quad x \mapsto (x,-1).
\]

\begin{theorem}
\label{th:FSnh}
  There are equivalences
  \begin{align*}
  \BDC(\field_\V) &\isoto \BDC_{\R^+,*\{t\geq0\}}(\field_{\Wt}),  & F&\mapsto
  (\reim i F)^\wedge,\\
  \BDC(\field_\V) &\isoto \BDC_{\R^+,\{t\geq0\}^*}(\field_{\Wt}),  & F&\mapsto
  (\roim i F)^\vee,
  \end{align*}
with quasi inverses $G\mapsto \opb i (G^\vee)[-1]$ and $G\mapsto \epb i (G^\wedge)[1]$, respectively.
\end{theorem}

The category $\BDC_{*\{t\geq0\}}(\field_{\Wt})$ is of the kind of categories
discussed by Tamarkin in~\cite{Tam10}. In Appendix~\ref{se:Tam} we
show how the above functor $F\mapsto (\reim i F)^\wedge$ is expressed
in terms of the Fourier transform considered in~\cite{Tam10}.

\begin{proof}
As the proofs are similar, we will only discuss the first equivalence.

By Proposition~\ref{pro:icone}, there is an equivalence
  \[
  \BDC(\field_\V) \isoto \BDC_{\R^+,\langle\{s<0\}
    \rangle}(\field_{\Vt}),
\quad F\mapsto (\reim i F)^\cone,
  \]
with quasi-inverse $\opb i[-1]$.
Since $(\reim i F)^\wedge \simeq (\reim i F)^{\cone\wedge}$,
by Theorem~\ref{thm:F}
we are left to prove that the Fourier-Sato transform between $\Vt$ and
$\Wt$ induces an equivalence
\begin{equation}
\label{eq:st}
\BDC_{\R^+,\langle\{s<0\} \rangle}(\field_{\Vt}) \isoto \BDC_{\R^+,*\{t\geq0\}}(\field_{\Wt}).
\end{equation}
By \eqref{eq:Fouetens} and Lemma~\ref{lem:Fex}~(i) one has
\[
\field_{\{s \geq 0\}}^\wedge \simeq (\field_V\etens\field_{\{s \geq 0\}})^\wedge
\simeq \field_V^\wedge\etens\field_{\{s \geq 0\}}^\wedge
\simeq \field_{\{t> 0,\,y=0\}}[-n]. 
\]
Let $H\in \BDC_{\R^+}(\field_{\Vt})$. By \eqref{eq:convF}, one has
\[
(H\tens\field_{\{s\geq 0\}})^\wedge \simeq H^\wedge \conv \field_{\{t> 0,\,y=0\}}[1].
\]
Hence the conditions $H\tens\field_{\{s\geq 0\}} =0$ and $H^\wedge
\conv \field_{\{t> 0,\,y=0\}} = 0$ are equivalent.
\end{proof}

\begin{remark}
It follows from \eqref{eq:wedgevee} that 
\[
(\roim i F)^\vee \simeq (\reim i F)^{\wedge a}[n+1].
\]
Thus, Theorem~\ref{th:FSnh} implies that for $G\in\BDC_{\R^+}(\field_{\Wt})$
the two conditions
\[
G \conv 
\field_{\{t> 0,\,y=0\}} =0, \qquad 
\hom^*(\field_{\{t< 0,\,y=0\}},G) =0,
\] 
are equivalent.
\end{remark}

\begin{remark}
\label{rem:SS}
One can recast the equivalence \eqref{eq:st} in terms of the theory of
microsupport from~\cite{KS90}.  Recall that the microsupport of
$F\in\BDC(\field_\V)$ is a closed conic involutive subset
$SS(F)\subset T^*\V$ of the cotangent bundle.  For $A\subset T^*\V$,
denote by $\BDC_{\mu\,A}(\field_\V)$ the full subcategory of
$\BDC(\field_\V)$ whose objects $F$ satisfy $SS(F)\subset A$. Denote
by $\pi\colon T^*\V\to \V$ the projection. Since $\supp(F) =
\pi(SS(F))$, for $S\subset \V$ one has $\BDC_S(\field_\V) =
\BDC_{\mu\,\opb\pi(S)}(\field_\V)$.

\smallskip\noindent(i)
From the adjunction isomorphism
\[
\Hom(F\tens\field_{\{s\geq 0\}}, F') \simeq \Hom(F, \rsect_{\{s\geq 0\}}(F'))
\]
one deduces that $\BDC_{\R^+,\langle\{s<0\} \rangle}(\field_{\Vt})$ is
the left orthogonal to $\BDC_{\R^+,\mu\,\{s\geq 0\}}(\field_{\Vt})$.

\smallskip\noindent(ii) Note that, using $t\in\R$ as coordinate, the
associated symplectic coordinates in $T^*\R$ are $(t;s)$. By
Tamarkin~\cite{Tam10}, $\BDC_{\R^+,*\{t\geq0\}}(\field_{\Wt})$ is the
left orthogonal to $\BDC_{\R^+,\mu\,\{s\leq 0\}}(\field_{\Wt})$.

\smallskip\noindent(iii)
The equivalence \eqref{eq:st} then follows from \cite[Theorem 5.5.5]{KS90}.
\end{remark}

\begin{lemma}
  \label{lem:FiF}
  \begin{itemize}
  \item[(i)] Consider the subset $\{\langle x,y
      \rangle \leq t\}\subset\V\times\Wt$.
For $F\in\BDC(\field_\V)$ one has
    \[
    (\reim i F)^\wedge \simeq F \comp[\V] \field_{\{\langle x,y
      \rangle \leq t\}}.
    \]
    In particular, $F^\wedge \simeq (\reim i
    F)^\wedge|_{\W\times\{0\}}$ and $(\reim i
    F)^\wedge|_{\{y=0,\,t<0\}} = 0$.
  \item[(ii)] For $F\in\BDC_{\R^+}(\field_\V)$ one has
    \[
    (\reim i F)^\wedge \simeq F^\wedge\etens\field_{\{t\geq 0\}}.
    \]
  \end{itemize}
\end{lemma}

\begin{proof}
  (i) is implied by the isomorphism
\[
\field_{\Gamma_i} \comp[\Vt] \field_{\{\langle x,y
  \rangle+st\leq0\}}\simeq\field_{\{\langle x,y \rangle\leq t\}}.
\]

  \medskip\noindent(ii) Recall that $(\reim i F)^\wedge \simeq (\reim i F)^{\cone\wedge}$.
By Lemma~\ref{lem:iFhom} one has
  \[
  (\reim i F)^\cone \simeq F \etens \field_{s<0}[1].
  \]
Taking the Fourier-Sato transform, the statement follows by
\eqref{eq:Fouetens} and Lemma~\ref{lem:Fex}~(ii).
\end{proof}

\medskip

We end this section by computing the non homogeneous Fourier transform
of $F=\field_A$ for some classes of locally closed subsets
$A\subset\V$.  Note that, by Lemma~\ref{lem:iAF}, one has
\begin{equation}
\label{eq:iAF}
(\reim i\field_A)^\cone \simeq \field_{\gamma_A}[1],
\end{equation}
where we denote by
\[
\gamma_A = \R^+(i(A)) \subset \Vt
\]
the cone generated by $i(A)$.

Let us first consider the case where $A$ is a nonempty, closed, convex
subset.  (For the notions of support function and asymptotic cone that
we now recall, see for example \cite{AT03}.)

The asymptotic cone of $A$ is defined by
\begin{align*}
\lambda_A &= \{x\in\V\colon a+\R^+x \subset A, \ \exists a\in A\} \\
&= \{x\in\V\colon a+\R^+x \subset A, \ \forall a\in A\}.
\end{align*}
It is the set of directions in which $A$ is infinite.
Under the identification $\V = \V\times\{0\}\subset \Vt$, one has
\begin{equation}
\label{eq:lambdaA}
\lambda_A =
\overline{\gamma_A} \cap (\V\times\{0\}),
\end{equation}
or equivalently $\overline{\gamma_A} = \gamma_A\sqcup\lambda_A$.

\begin{lemma}
The cone $\overline{\gamma_A}$ is proper if and only if $A$ contains no affine line.
\end{lemma}

\begin{proof}
It follows from \eqref{eq:lambdaA} and the definition of $\lambda_A$, by noticing that $\overline{\gamma_A} \subset \{s\leq 0\}$.
\end{proof}

The support function of $A$ is defined by
\[
\suppfct_A \colon \W \to[] \R\cup\{+\infty\}, \quad y \mapsto \sup_{x\in A}\langle x,y \rangle.
\]
It describes the signed distance from the origin of the supporting hyperplanes of $A$.
Recall that $\suppfct_A$ is positive homogeneous, lower semicontinuous
and convex. Moreover, its effective domain (that is, the set of
$y\in\W$ such that $\suppfct_A(y) < +\infty$) is $\lambda_A^{\circ a}$.

\begin{lemma}
  \label{lem:gammaHK}
  One has
  \[
  \gamma_A^\circ = \{(y,t)\in\Wt\colon y\in \lambda_A^\circ,\ t \leq
  -\suppfct_A(-y)\}.
  \]
\end{lemma}

\begin{proof}
  By definition,
  \[
  \gamma_A^\circ = \{(y,t)\in\Wt\colon t \leq \langle x,y \rangle,\
  \forall x\in A \}.
  \]
  It is then enough to note that $\inf_{x\in A}\langle x,y \rangle =
  -\suppfct_A(-y)$ and to recall that
  $-\suppfct_A(-y) = -\infty$ for $y\notin \lambda_A^\circ$.
\end{proof}

Consider the projection $q_1\colon \Wt = \W\times\R \to \W$.

\begin{lemma}\label{lem:conefou}
\begin{itemize}
\item[(i)]
Let $A\subset\V$ be a nonempty, closed, convex
    subset which contains no affine line. Then  
\[
    (\reim i\field_A)^\wedge \simeq \opb{q_1}\field_{\Int\lambda_A^\circ} \tens \field_{\{t \geq
  -\suppfct_A(-y)\}}.
\]
\item[(ii)]
Let $A\subset\V$ be an nonempty, open, convex
    subset. Then  
\[
    (\reim i\field_A)^\wedge \simeq
    \opb{q_1}\field_{\Int\lambda_A^{\circ a}} \tens \field_{\{t \geq
  \suppfct_A(y)\}}[-n].
\]
\end{itemize}
\end{lemma}

\begin{proof}
  (i) Note that $\overline{\gamma_A}$ is a proper
  closed convex cone containing the origin and $\gamma_A = \overline{\gamma_A} \cap \{s<0\}$.  
By \eqref{eq:iAF}
  and Lemma~\ref{lem:Fex}~(i), we have
\[
(\reim i\field_A)^\wedge \simeq \field_{\gamma_A}^\wedge[1] \simeq
(\field_{\overline{\gamma_A}} \tens \field_{\{s<0\}})^\wedge[1]
\simeq \field_{\Int \gamma_A^\circ} \conv \field_{\{t\geq0,\,y=0\}}[1].
\]
By Lemma~\ref{lem:gammaHK},
\[
\Int \gamma_A^\circ = \{(y,t)\in\Wt\colon y\in \Int\lambda_A^\circ,\ t <
  -\suppfct_A(-y)\}.
\]
Then one has
\[
\field_{\Int \gamma_A^\circ} \conv \field_{\{t\geq0,\,y=0\}} \simeq 
\field_{\{y\in \Int\lambda_A^\circ,\ t \geq
  -\suppfct_A(-y)\}}[-1].
\]

\medskip\noindent (ii)
By \eqref{eq:iAF}, Lemma~\ref{lem:Fex} and \eqref{eq:convF}, we have
\[
(\reim i\field_A)^\wedge \simeq \field_{\gamma_A}^\wedge[1] \simeq
\field_{\Int \gamma_A^{\circ a}}[-n],
\]
and one concludes by Lemma~\ref{lem:gammaHK}.
\end{proof}

Let us now treat a non convex case.  We consider the geometric
situation of Lemma~\ref{lem:qcone}, so that $(x',x'',x''')$ is the
coordinate system on $\V = \R^p \times \R^q \times \R^r$, and
$(y',y'',y''')$ is the dual coordinate system on $\W$.

\begin{lemma}
\label{le:quadric}
For $c \geq 0$, consider the quadric
\[
A = \{(x',x'')\in\V\colon x^{\prime 2}-x^{\prime\prime 2} \leq c^2,\ x'''=0\}
\]
and set
\[
g(y) = 
\begin{cases}
c \sqrt{y^{\prime 2}-y^{\prime\prime 2}}, &\text{for }y^{\prime 2}-y^{\prime\prime 2} \geq 0,\\
0,&\text{else}.
\end{cases}
\]
Then
\[
 (\reim i\field_A)^\wedge \simeq \opb{q_1} \field_{\{y^{\prime
     2}-y^{\prime\prime 2} \geq 0\}} \tens \field_{\{t \geq -g(y)\}}[-q].
\]
\end{lemma}

\begin{proof}
For $c=0$ the sheaf $\field_A$ is conic. The statement then follows
from Lemmas~\ref{lem:qcone} and \ref{lem:FiF}.

For $c>0$ one has $\gamma_A = \{x^{\prime 2}-x^{\prime\prime 2} \leq c^2 s^2\} \cap \{s<0\}$. 
By \eqref{eq:iAF}, Lemma~\ref{lem:qcone} and \eqref{eq:convF}, it
follows that
\begin{multline*}
(\reim i\field_A)^\wedge \simeq \field_{\gamma_A} ^\wedge[1] \simeq
(\field_{\{x^{\prime 2}-x^{\prime\prime 2} \leq c^2 s^2\}} \tens
\field_{\{s<0\}})^\wedge[1] \\ 
\simeq \field_{\{y^{\prime 2}-y^{\prime\prime 2} \geq (1/c^2) t^2\}}
\conv 
\field_{\{t\geq0,\,y=0\}}[-q].
\end{multline*}
Since
\[
\{y^{\prime 2}-y^{\prime\prime 2} \geq (1/c^2) t^2\} = \{y^{\prime
  2}-y^{\prime\prime 2} \geq 0,\ |t| \leq c \sqrt{y^{\prime 2}-y^{\prime\prime 2}}\},
\]
one has
\[
\field_{\{y^{\prime 2}-y^{\prime\prime 2} \geq (1/c^2) t^2\}} \conv \field_{\{t\geq0,\,y=0\}} \simeq
\field_{\{y^{\prime 2}-y^{\prime\prime 2} \geq 0,\ t\geq-c \sqrt{y^{\prime 2}-y^{\prime\prime 2}}\}}
\]
\end{proof}

\section{Exponential growth: real case}\label{se:expR}

Here, in order to treat functions with exponential growth, we will generalize
the construction of the sheaf of tempered functions of~\cite{KS01} (see also~\cite{Pre08}).

\medskip

Let $X$ be a real analytic manifold. 
From now on we set $\field=\C$.
Denote by $\BDC_{\Rc}(\field_X)$
the full triangulated subcategory of  $\BDC(\field_X)$ whose objects
have $\R$-constructible cohomology groups. 

Denote by $X_\sa$ the subanalytic site. This is the site
whose objects are open subanalytic subsets of $X$ and whose
coverings are locally finite in $X$. One calls subanalytic sheaf on $X$
a sheaf on $X_\sa$. 

Consider the natural map
\[
\rho \colon X \to X_\sa.
\]
Besides the usual right adjoint $\oim\rho$, the pull-back
functor $\opb\rho$ has a left adjoint $\eim\rho$.
The push-forward $\oim\rho$ induces a fully faithful exact functor from $\R$-constructible sheaves to
subanalytic sheaves. One thus identifies $\BDC_{\Rc}(\field_X)$ as a
full triangulated
subcategory of $\BDC(\field_{X_\sa})$.

Denoting by $\Db_X$ the sheaf of Schwartz's distributions, the
subanalytic sheaf $\Db_X^t$ of tempered distributions is defined by
\[
\Db_X^t(U) = \Db_X(X) / \sect_{X\setminus U}(X; \Db_X)
\]
for $U\subset X$ an open subanalytic subset. The sheaf
$\Db_X^t$ is acyclic on $X_\sa$.

One says that a function $\varphi$ on $U$ has polynomial growth
at $x_\circ\in X$ if it satisfies the following condition. For a local
coordinate system at $x_\circ$, there exist a
sufficiently small compact neighborhood $K$ of $x_\circ$ and constants
$c\geq0$, $m\in\Z_{>0}$ such that
\begin{equation}
\label{eq:poly}
|\varphi(x)| \leq 
c  \left( 1+ \frac1{\dist(K\setminus U,x)} \right)^m, \quad \forall
x\in K \cap U,
\end{equation}
where ``$\dist$'' denotes the euclidean distance on the domain of the
coordinates.

One says that $\varphi$ has polynomial growth on $X$ if it has
polynomial growth at any $x_\circ\in X$.

One says that $\varphi \in \shc_X^{\infty}(U)$ is tempered
if all of its derivatives have polynomial growth.

The subanalytic sheaf of tempered smooth functions is defined by
\[
\cit_X \colon U \mapsto \{\varphi\in\shc_X^{\infty}(U) \colon \varphi \text{ is tempered}\}.
\]
It is an acyclic sheaf on $X_\sa$.

Denote by $\shd_X$ the ring of analytic finite order differential
operators. Since sections of $\oim\rho\shd_X$ do not take growth conditions into account, the sheaves $\Db_X^t$ and $\cit_X$ are not $\oim\rho\shd_X$-modules. However, they are $\eim\rho\shd_X$-modules.

Recall that a function $\delta\colon X\to \R$ is called subanalytic if
its graph is a subanalytic subset of $X\times\R$.
By \L ojasiewicz inequalities one has

\begin{lemma}
\label{lem:rho}
The estimate \eqref{eq:poly} is equivalent to
\[
|\varphi(x)| \leq 
c  \left( 1+ \frac 1{\delta(x)} \right)^m, \quad \forall
x\in K \cap U,
\]
for $\delta\geq 0$ a continuous subanalytic function on $K$ such that
$K\cap\partial U = \{\delta(x)=0\}$.
\end{lemma}

\medskip

Let $j\colon X \to X'$ be an open subanalytic embedding of real analytic
manifolds. Denote by $X_{j\text-\sa}$ the site structure induced on
$X$ by $X'_\sa$. This is the site whose objects are open subsets of $X$ which are
subanalytic in $X'$ and whose coverings are locally finite in
$X'$. Let us call $j$-subanalytic such open subsets.

Let us say that a sheaf $F$ on $X$ is $j$-$\R$-constructible if $\reim j F$ is
$\R$-constructible in $X'$. Denote by $\BDC_{j\text-\Rc}(\field_X)$
the full triangulated category of $\BDC(\field_X)$ whose objects have 
 $j$-$\R$-constructible cohomology groups. We identify
 $\BDC_{j\text-\Rc}(\field_X)$ to a full triangulated
 subcategory of $\BDC(\field_{X_{j\text-\sa}})$.

The following sheaves on $X_{j\text-\sa}$ take into account growth
conditions at infinity
\[
\Db_{X|X'}^t = \left.\Db_{X'}^t\right|_{X_{j\text-\sa}}, \quad 
\cit_{X|X'} = \left.\cit_{X'}\right|_{X_{j\text-\sa}}.
\]
Note that these are not sheaves of $\eim\rho\shd_X$-modules, but modules over the
ring $\eim\rho\shd_{X'}|_{X_{j\text-\sa}}$.

Set
\[
\widetilde X = X\times\R, \quad \widetilde X' = X' \times \PR(\R),
\]
where $\PR(\R)$ denotes the real projective line.

A function $f\colon X \to \R$ is called $j$-subanalytic if its graph
is subanalytic in $\widetilde X'$. 
Note that by \L ojasiewicz inequalities such an $f$ has polynomial growth.

Let $f\colon X \to \R$ be a continuous $j$-subanalytic function and
$U\subset X$ an open $j$-subanalytic subset.
One says that a function $\varphi$ on $U$ has $f$-exponential growth
at $x_\circ\in X'$ if it satisfies the following condition. For a local
coordinate system at $x_\circ$, there exist a
sufficiently small compact neighborhood $K$ of $x_\circ$ and constants $c\geq0$, $m\in\Z_{>0}$ such that
\[
|\varphi(x)| \leq 
c  \left( 1+ \frac1{\dist(K\setminus U,x)} + |f(x)| \right)^m e^{f(x)}, \quad \forall
x\in K \cap U.
\]
Note that one gets an equivalent definition by replacing the function
$\dist(K\setminus U,\cdot)$ with a subanalytic function $\delta$ as in
Lemma~\ref{lem:rho}.

One says that $\varphi$ has $f$-exponential growth on $X'$ if
 it has $f$-exponential growth at any $x_\circ\in X'$.

One says that $\varphi \in \shc_X^{\infty}(U)$ is $f$-tempered
if all of its derivatives have $f$-exponential growth.

\begin{definition}
\label{def:expdummy}
The presheaf of $f$-tempered smooth functions on the site $X_{j\text-\sa}$ is defined by
\[
\cit[f]_{X|X'} \colon U \mapsto \{\varphi\in\shc_X^{\infty}(U) \colon \varphi \text{ is $f$-tempered}\}.
\]
It is a presheaf of $\eim\rho\shd_{X'}|_{X_{j\text-\sa}}$-modules.
\end{definition}

Note that ``$(f+c)$-tempered'' is the same as ``$f$-tempered'' for $c\in\R$.
One has
\[
\cit_{X|X'} = \cit[0]_{X|X'}.
\]

Let us show how $f$-tempered functions are related with tempered
functions with one additional variable.

Denote by $q_1\colon \widetilde X = X\times\R \to X$  the projection. 
Let $s\in\R$ be the coordinate and denote by $D_\R$ the ring of
differential operators with polynomial coefficients.

\begin{proposition}
\label{pro:expdummy}
Let $f\colon X \to \R$ be a continuous $j$-subanalytic function.
There are isomorphisms
\begin{align*}
\cit[f]_{X|X'}&\simeq \roim{q_1}\rhom(\field_{\{s < -f(x)\}},
\rhom[D_\R](D_\R e^s,\cit_{\widetilde X|\widetilde
  X'})) \\
&\simeq \roim{q_1}\rhom(\field_{\{s \geq -f(x)\}},
\rhom[D_\R](D_\R e^s,\cit_{\widetilde X|\widetilde
  X'}))[1], \\
\cit_{X|X'}&\simeq \roim{q_1}\rhom[D_\R](D_\R e^{is},\cit_{\widetilde X|\widetilde X'}).
\end{align*}
In particular, the complexes on the right hand side are concentrated in
degree zero and the presheaf $\cit[f]_{X|X'}$ is an acyclic sheaf.
\end{proposition}

\begin{proof}
(i) Let us prove the first isomorphism.
Set
\[
\shc = \roim{q_1}\rhom(\field_{\{s < -f(x)\}},
\rhom[D_\R](D_\R e^s,\cit_{\widetilde X|\widetilde
  X'})).
\]
For $U\subset X$ a $j$-subanalytic open subset, one has
\begin{equation}
\label{eq:CUE}
\rsect(U;\shc) \simeq (E \to[\partial_s-1] E),
\end{equation}
where the complex on the right hand side is in degrees 0 and 1, and
\[
E = \sect( \opb{q_1}U \cap \{s < -f(x)\} ;\cit_{\widetilde X|\widetilde X'}).
\]

\smallskip\noindent (i-a)
To prove that $\shc$ is concentrated in degree zero,
it is enough to show the surjectivity of $\partial_s-1$ in
\eqref{eq:CUE}.

Let $\gamma\colon X\to\R$ be a $C^\infty$  function such that
$-2<f+\gamma<-1$.
For $\Psi\in E$, a $C^\infty$ solution $\Phi$ to $(\partial_s-1)\Phi =
\Psi$ is given by
\begin{equation}
\label{eq:Phi}
\Phi(x,s) = e^s \int_{\gamma(x)}^s e^{-u}\Psi(x,u)\,du.
\end{equation}
We are thus left to prove that $\Phi\in E$. Since the estimates for
the derivatives of $\Phi$ are similar, let us only show that 
$\Phi$ has polynomial growth. Since $\Psi$ has polynomial growth,
by Lemma~\ref{lem:rho} any $x_\circ\in X'$ has a compact neighborhood $K$
such that there are constants $c,m$ with 
\begin{multline*}
|\Psi(x,s)| \leq 
c  \left( 1+ \frac1{\dist(K\setminus U,x)} + |s| + \frac 1{|s+f(x)|} \right)^m, \\ \forall
x\in K \cap U,\ \forall s<-f(x).
\end{multline*}
Then \eqref{eq:Phi} implies
\[
|\Phi(x,s)| \leq c e^s \left|\int_{\gamma(x)}^s e^{-u} \left(
      1+ \frac1{\dist(K\setminus U,x)} + |u| + \frac 1{|u+f(x)|} \right)^m\,du\right|.
\]
One thus deduces that $\Phi$ has polynomial growth from
Lemma~\ref{lem:est} below.

\noindent (i-b)
By \eqref{eq:CUE}, an element of
$H^0\rsect(U;\shc)$ is a solution $\Phi\in E$ of the equation
$(\partial_s-1)\Phi = 0$. Thus $\Phi(x,s) = \varphi(x)e^s$.
The map
\[
\cit[f]_{X|X'}(U) \to H^0\rsect(U;\shc), \quad
\varphi(x)\mapsto \varphi(x)e^s
\]
is well defined. To show that it is an isomorphism, we have to prove that
it is surjective. Given $\Phi(x,s) = \varphi(x)e^s$ with $\Phi\in E$,
there is an estimate
\begin{multline*}
|\varphi(x) |e^s \leq 
c  \left( 1+\frac1{\dist(K\setminus U,x)} + |s| + \frac1{|s+f(x)|} \right)^m , \\ \forall
x\in K \cap U,\ \forall s<-f(x).
\end{multline*}
Taking $s=-f(x)-1$, we see that $\varphi$ has $f$-exponential growth.
A similar argument holds for the derivatives of $\Phi$.

\noindent (ii)
The second isomorphism in the statement follows from the first one if
one shows that
\[
\roim{q_1}\rhom[D_\R](D_\R e^s,\cit_{\widetilde
  X|\widetilde
  X'}) = 0.
\]
This is proved in a similar way to part (i) above.

\noindent (iii)
The proof of the third isomorphism in the statement is again similar
to part (i) above taking $\gamma=0$.
\end{proof}

\begin{lemma}
\label{lem:est}
Let $f\colon X\to\R$ be a continuous $j$-subanalytic function and
$\gamma\colon X\to\R$ a $C^\infty$  function such that
$-2<f+\gamma<-1$. Then, for any $m,m'\in\Z_{\geq 0}$, the function defined for $x\in
X$, $s<-f(x)$ by
\[
e^s\int_{\gamma(x)}^s\frac{|u|^m e^{-u}}{|u+f(x)|^{m'}}du
\]
has polynomial growth on $\widetilde X'$.
\end{lemma}

\begin{proof}
Recall that $f(x)$ has polynomial growth. By the estimate $|u| \leq
|u+f(x)| + |f(x)|$, one reduces to prove that for $m\in\Z_{\geq 0}$
the functions
\[
\Phi(x,s) = e^s\int_{\gamma(x)}^s |u|^m e^{-u} du,
\qquad
\Psi(x,s) = e^s\int_{\gamma(x)}^s\frac{e^{-u}}{|u+f(x)|^{m+1}}du,
\]
have polynomial growth. 

Recall that $\int u^m e^{-u} du = P(u) e^{-u}$ for
$P$ a polynomial of degree $m$. Then
\[
\left| \Phi(x,s) \right| 
\leq
c (1+ |s|^m + |\gamma(x)|^m e^{s+\gamma(x)}).
\]
Since $\gamma$ has polynomial growth and $s+\gamma(x)<0$, we deduce
that $\Phi$ has polynomial growth.

Since $-2<f+\gamma<-1$ and $s+f(x)<0$, we have
\begin{align*}
  \left|\Psi(x,s)\right| &=
  \left|e^{s+f(x)}\int_{\gamma(x)+f(x)}^{s+f(x)}\frac{e^{-u}}{|u|^{m+1}}du\right| \\
  &\leq c
  \left(1+e^{s+f(x)}\left|\int_{-1}^{s+f(x)}\frac{e^{-u}}{|u|^{m+1}}du\right|\right).
\end{align*}
From the estimate $e^{-u}/|u|^{m+1} \leq e^{-u} + 1/|u|^{m+2}$ ($u<0$), we finally get
\[
\left|\Psi(x,s)\right| \leq
c' \left(1+\frac{1}{|s+f(x)|^{m+1}}\right).
\]
\end{proof}

\section{Exponential growth: complex case}\label{se:expC}

Let $X$ be a complex analytic manifold.
Denote by $\sho_X$ the sheaf of holomorphic functions and
by $\shd_X$ the ring of holomorphic finite order differential operators.

Denote by $\overline X$ the conjugate complex manifold to $X$,
so that sections of $\sho_{\overline X}$ are conjugates of sections of $\sho_X$.
The real analytic manifold underlying $X$ is identified with the diagonal of $X\times\overline X$. 
By Dolbeault resolution, one has
\[
\sho_X= \rhom[\shd_{\overline X}|_X](\sho_{\overline X}|_X, \shc^{\infty}_{X_\R}).
\]
Similarly, following \cite{KS01}, the complex of tempered holomorphic
functions is defined by
\[
\sho^t_X = \rhom[\shd_{\overline X}|_X](\sho_{\overline X}|_X, \cit_{X_\R}).
\]

Let $j\colon X \to X'$ be an open subanalytic embedding of complex analytic
manifolds. 
For $f\colon X\to\R$ a continuous $j$-subanalytic function,
let
\[
\sho_{X|X'}^{\temp f} = \rhom[\eim\rho\shd_{\overline
  X'}|_{X_{j\text-\sa}}](\eim\rho\sho_{\overline
  X'}|_{X_{j\text-\sa}},\cit[f]_{X|X'})
\]
be the Dolbeault complex of $\cit[f]_{X|X'}$.
In particular, $\sho_{X}^t = \sho_{X|X}^{\temp 0}$. Set
\[
\sho_{X|X'}^t =
\sho_{X|X'}^{\temp 0}.
\]
Recall that if $X$ is a complexification of
$M$, then
\begin{equation}
\label{eq:Dbtmm'}
\Db_{M|M'}^t \simeq \rhom(F,\sho^t_{X|X'}),
\end{equation}
where $M'$ is the closure of $M$ in $X'$, $F =
\rhom(\field_M,\field_X)\simeq or_M[-n]$ and $or_M$ is the
orientation sheaf and $n$ the dimension of $M$.

Denote by $\PC(\C)$ the complex projective line. Set
\[
\widetilde X = X\times\C, \quad \widetilde X' = X' \times \PC(\C)
\]
and denote by $q_1\colon \widetilde X \to X$ the projection. 

\begin{proposition}
\label{pr:Lapnh}
There are isomorphisms in $\BDC(\eim\rho\shd_{\overline
  X'}|_{X_{j\text-\sa}})$
\begin{align*}
\sho_{X|X'}^{\temp f} &\simeq
\roim{q_1}\rhom(\field_{\{\Re\,s<-f(x)\}} , \rhom[D_\C](D_\C e^s,
\sho_{\widetilde X|\widetilde X'}^t)) \\
&\simeq \roim{q_1}\rhom(\field_{\{\Re\,s\geq -f(x)\}} , \rhom[D_\C](D_\C e^s,
\sho_{\widetilde X|\widetilde X'}^t))[1].
\end{align*}
\end{proposition}

\begin{proof}
(i)
Let us prove the first isomorphism.
One has
\begin{align*}
  \sho_{\widetilde X|\widetilde X'}^t & =
  \rhom[\eim\rho\shd_{\overline {\widetilde X}'}|_{\widetilde
    X_{j\text-\sa}}](\eim\rho\sho_{\overline {\widetilde
      X}'}|_{\widetilde X_{j\text-\sa}},\cit[f]_{\widetilde
    X|\widetilde X'}) \\
  &\simeq \rhom[\eim\rho\shd_{\overline
    X'}|_{X_{j\text-\sa}}](\eim\rho\sho_{\overline
    X'}|_{X_{j\text-\sa}},\rhom[D_{\overline \C}](D_{\overline
    \C}/\langle\overline\partial_s\rangle, \cit[f]_{\widetilde
    X|\widetilde X'})),
\end{align*}
where $\langle\overline\partial_s\rangle \subset D_{\overline \C}$
denotes the left ideal generated by
$\langle\overline\partial_s\rangle$.  It is then enough to prove the
isomorphism
\begin{multline*}
  \cit[f]_{X|X'} \simeq \\
  \roim{q_1}\rhom(\field_{\{\Re\,s<-f(x)\}} , \rhom[D_\C\etens
  D_{\overline \C}](D_\C e^s \etens D_{\overline
    \C}/\langle\overline\partial_s\rangle, \cit_{\widetilde
    X|\widetilde X'})).
\end{multline*}
In the identification $\C=\R\times\R$ given by $s=\lambda+i\mu$, there
is an isomorphism of $D_{\R\times\R}$-modules
\[
D_\C e^s
\etens D_{\overline \C}/D_{\overline \C}\,\overline\partial_s \simeq D_\R e^{\lambda} \etens
D_\R e^{i\mu}.
\] 
Moreover, one has
\[
\cit_{\widetilde X|\widetilde X'} \simeq \cit_{X\times\R\times\R|X\times\PR(\R)\times\PR(\R)}.
\] 
The statement then follows from Proposition~\ref{pro:expdummy}.

\medskip\noindent(ii)
The proof of the second isomorphism is similar.
\end{proof}

\begin{remark}
It would be interesting to consider also other growth conditions, like for example those used in~\cite{Kaw70} to construct Fourier hyperfunctions.
\end{remark}

Consider the closed embedding
\[
i\colon X\to \widetilde X, \quad x \mapsto (x,-1). 
\]
It follows from \cite[Theorem~7.4.6]{KS01} that one has
\begin{equation}
\label{eq:iOt}
\roim i\sho^t_{X|X'} \simeq
\rhom[D_\C](D_\C\delta(s+1), \sho^t_{\widetilde X|\widetilde X'})[1],
\end{equation}
where $D_\C\delta(s+1) = D_\C/D_\C(s+1)$ is the $D_\C$ module generated
by the $\delta$ function of $s=-1$.

\section{Laplace transform}\label{se:L}
We recall here a theorem of~\cite{KS97} on the
Fourier-Laplace transform between tempered holomorphic functions associated
with conic sheaves on dual complex vector spaces. 

\medskip

Let $\VV$ and $\WW$ be dual complex $n$-dimensional vector spaces by
the complex pairing $(x,y)\mapsto \langle x,y  \rangle$.
Denote by $\PC(\VV)$ and $\PC(\WW)$ the complex projective
compatifications of $\VV$ and $\WW$, respectively.
Let $j\colon \VV \to \PC(\VV)$ and $j\colon\WW\to\PC(\WW)$ be the embeddings.

Denote by $D_\VV$ the Weyl algebra and by
\[
(\cdot)^\wedge\colon D_\VV \to D_{\WW}
\]
the Fourier isomorphism. If $(x_1,\dots,x_n)$ is a coordinate system
on $\VV$ and $(y_1,\dots,y_n)$ the dual coordinate system on $\WW$,
this is given by
\[
x_i^\wedge = -\partial_{y_i}, \quad \partial_{x_i}^\wedge = y_i. 
\]
If $N$ is a $D_\VV$-module, denote by $N^\wedge$ the vector space $N$
endowed with the $D_{\WW}$-module structure induced by $\wedge$.

Note that the Fourier-Sato transform between $\VV$ and $\WW$ is associated with the
kernel $\field_{\{\Re\langle x,y\rangle\leq0\}}$.

A result linking the Laplace and Fourier-Sato transform was
established in \cite{Mal88}. This was reconsidered and generalized
in~\cite{KS97}, whose main result describes the Laplace transform of
conic tempered holomorphic functions:

\begin{theorem}[{\cite[Theorem~5.2.1]{KS97}}]
\label{thm:KSL}
Let $F\in\BDC_{\R^+,j\text-\Rc}(\field_\VV)$.
The Laplace transform $\varphi\mapsto\int\varphi(x)e^{-\langle
x,y\rangle} dx$
induces an isomorphism in $\BDC(D_{\WW})$
  \[
  \RHom(F,\sho_{\VV|\PC(\VV)}^t)^\wedge \simeq 
\RHom(F^\wedge[n],\sho_{\WW|\PC(\WW)}^t).
  \]
\end{theorem}

In particular, for $N\in\BDC(D_{\WW})$ one has
\begin{multline}
\label{eq:KSL}
  \RHom(F,\rhom[D_\VV](N^\wedge,\sho_{\VV|\PC(\VV)}^t)) \simeq \\
\RHom(F^\wedge[n],\rhom[D_{\WW}](N,\sho_{\WW|\PC(\WW)}^t)).
\end{multline}

\begin{remark}
\label{rem:Pre}
As shown in~\cite{Pre11}, Theorem~\ref{thm:KSL} is reformulated
in the framework of the conic subanalytic site by the
isomorphism
\[
(\sho_\VV^{t,c})^\wedge \simeq \sho_{\WW}^{t,c}[-n],
\]
where $\sho_\VV^{t,c}$ denotes the complex of conic tempered
holomorphic functions.
\end{remark}

\section{Conified Laplace transform}\label{se:coL}
Here, we extend Theorem~\ref{thm:KSL} to sheaves which
are not necessarily conic.

\medskip

As in the previous section, let $\VV$ and $\WW$ be dual complex
$n$-dimensional vector spaces by
the complex pairing $(x,y)\mapsto \langle x,y  \rangle$.
Recall that we denote by $j$ the embeddings $\VV\subset\PC(\VV)$ and $\WW\subset\PC(\WW)$.

Consider the dual vector spaces
\[
\VVt = \VV\times\C, \qquad \WWt = \WW \times \C
\]
by the complex pairing $\langle(x,s),(y,t)\rangle = \langle x,y
\rangle + st$. 

In order to extend Theorem~\ref{thm:KSL} to the case of not
necessarily conic sheaves, consider the embedding
\[
i\colon \VV \to \VVt, \quad x \mapsto (x,-1).
\]

Let $g\colon\WW\to\R$ be a continuous $j$-subanalytic function,
positive homogeneous of degree one.
Let $F\in\BDC_{j\text-\Rc}(\field_\VV)$ and
$G\in\BDC_{\R^+,j\text-\Rc}(\field_{\WW})$.
Assume that
\begin{equation}
\label{eq:FfouG}
(\reim i F)^\wedge \simeq
(G\etens\field_\C)\tens\field_{\{\Re t \geq -g(y)\}}.
\end{equation}
Note that Conjecture~\ref{con:Tam} below suggests that this assumption
is not so strong.

\begin{theorem}
\label{thm:NHF}
Assume \eqref{eq:FfouG}.  The Laplace transform
$\varphi\mapsto\int\varphi(x)e^{-\langle x,y\rangle} dx$ induces an
isomorphism in $\BDC(D_{\WW})$
\[
\RHom(F,\sho^t_{\VV|\PC(\VV)}) \simeq 
\RHom(G[n],\sho^{\temp g}_{\WW|\PC(\WW)}).
\]
\end{theorem}

\begin{proof}
By \eqref{eq:FfouG} and Proposition~\ref{pr:Lapnh}, one has
\[
\RHom(G[n],\sho^{\temp g}_{\WW|\PC(\WW)}) \simeq
\RHom((\reim i  F)^\wedge[n-1],\rhom[D_\C](D_\C e^t,\sho_{\WWt|\PC(\WWt)}^t))
\]
The statement then follows from the chain of isomorphisms
\begin{align*}
  \RHom(F,& \sho_{\VV|\PC(\VV)}^t) \\
  &\simeq \RHom(\opb i((\reim i
  F)^\cone)[-1],\sho_{\VV|\PC(\VV)}^t) \\
  &\simeq \RHom((\reim i
  F)^\cone[-1],\roim i\sho_{\VV|\PC(\VV)}^t) \\
  &\simeq \RHom((\reim i F)^\cone[-2],\rhom[D_\C](
  D_\C\delta(s+1),\sho_{\VVt|\PC(\VVt)}^t)) \\
  &\simeq \RHom((\reim i F)^\wedge[n-1],\rhom[D_\C](D_\C
  e^t,\sho_{\WWt|\PC(\WWt)}^t)).
\end{align*}
The first isomorphism follows from \eqref{eq:tempisos}.  The third one
follows from \eqref{eq:iOt}. The last one follows from
\eqref{eq:KSL}. In fact, since $(s+1)^\wedge = -\partial_t+1$, one has
$D_\C\delta(s+1) \simeq (D_\C e^t)^\wedge$.
\end{proof}

\begin{remark}
With notations as in Remark~\ref{rem:Pre}, in the conic subanalytic
framework one has
\[
(\roim i\sho_\VV^{t,c})^\wedge \simeq \rhom[D_\C](D_\C e^t,\sho_{\WWt}^{t,c})[-n].
\]
\end{remark}

\section{Paley-Wiener type theorems}\label{se:PW}

As an application of Theorem~\ref{thm:NHF}, we obtain here
some Paley-Wiener type theorems.

\medskip

Recall that $\lambda_A$ and $\suppfct_A$ denote the asymptotic cone
and the support function of a convex subset $A\subset\V$. 
The function $\suppfct_A$ is continuous on
$\Int\lambda_A^\circ$, and is also subanalytic if so is $A$.

\begin{corollary}\label{cor:PW}
\begin{itemize}
\item[(i)] Let $A\subset \VV$ be a nonempty, closed, subanalytic,
  convex subset which contains no affine line. The Fourier-Laplace
  transform induces an isomorphism
  \begin{equation}
\label{eq:PWi}
    \rsect_A(\VV;\sho_{\VV|\PC(\VV)}^t)[n] \isoto
    \rsect(\Int\lambda_A^\circ,\sho_{\WW|\PC(\WW)}^{\temp{\suppfct_A(-y)}}),
  \end{equation}
and these complexes are concentrated in degree zero.

\item[(ii)]
  Let $A\subset \VV$ be a nonempty, open, subanalytic, convex
  subset. The Fourier-Laplace transform induces an isomorphism
   \begin{equation}
\label{eq:PWii}
    \rsect(A;\sho_{\VV|\PC(\VV)}^t) \isoto
    \rsect_{\lambda_A^\circ}(\WW,\sho_{\WW|\PC(\WW)}^{\temp{-\suppfct_A(y)}})[n],
  \end{equation}
and these complexes are concentrated in degree zero.
\end{itemize}
\end{corollary}

Note that if $A$ is bounded, then $\lambda_A^\circ = \WW$.

\begin{remark}
Here we are considering the Laplace transform with kernel
  $e^{-\langle x, y \rangle}$.  For the transform with kernel
  $e^{i\langle x, y \rangle}$, one should read $iy$ instead of
  $y$ in the above statement.
\end{remark}

\begin{remark}
  It would be interesting to relate isomorphisms \eqref{eq:PWi} and
  \eqref{eq:PWii} with the ones induces by the Radon transform as
  in~\cite{DS96}. (For a link between Radon and Fourier transforms
  see~\cite{DE03}.)
\end{remark}

\begin{proof}
  The fact that the complexes are concentrated in degree zero follows
  from~\cite[Theorem 5.10]{DS96}.

Decompose the embedding $i\colon \VV \to \VVt$ as
\[
\VV \to[i_\R] \VV\times\R \to[\ell] \VV\times\C = \VVt,
\]
where $i_\R(x) = (x,-1)$ and $\ell$ is induced by the embedding
$\R\subset\C$.  Note that the transpose ${}^t\ell\colon \WWt =
\WW\times\C \to \WW\times\R$ is induced by the projection $\C\to\R$,
$t\mapsto \Re\,t$.
For $A\in\VV$ a locally closed subset, by \eqref{eq:Foueim} and \eqref{eq:iAF} one has 
\begin{equation}
\label{eq:ell}
  (\reim i \field_{A})^\wedge \simeq (\reim \ell \reim{i_\R}
  \field_{A})^\wedge \simeq \opb{{}^t\ell} ((\reim{i_\R}
  \field_{A})^\wedge).
\end{equation}

\medskip\noindent(i)
By \eqref{eq:ell} and Lemma~\ref{lem:conefou}~(i) we have
\[
  (\reim i\field_A)^\wedge \simeq (\field_{\Int\lambda_A^\circ}\etens\field_\C) 
\tens \field_{\{\Re t \geq  -\suppfct_A(-y)\}}.
\]
Hence \eqref{eq:PWi} follows from Theorem~\ref{thm:NHF}.

\medskip\noindent(ii)
By \eqref{eq:ell} and Lemma~\ref{lem:conefou}~(ii) we have
  \[
  (\reim i\field_A)^\wedge \simeq (\field_{\lambda_A^{\circ a}}\etens\field_\C) 
\tens \field_{\{\Re t \geq  \suppfct_A(y)\}}[-2n].
  \]
Since $A$ is relatively compact, $\lambda_A^{\circ a} = \WW$.
Hence \eqref{eq:PWii} follows from Theorem~\ref{thm:NHF}.
\end{proof}

Let us describe some particular cases.
Assume that $\VV$ and $\WW$ are complexifications of $\V$ and $\W$, respectively. 
Denote by $\PR(\V)$ and $\PR(\W)$ the real
projective compactifications of $\V$ and $\W$, respectively.
If $A\subset\V$ is a closed subanalytic subset, one has
  \[
\sect_A(\V;\Db_{\V|\PR(\V)}^t) \simeq \rsect_A(\VV;\sho_{\VV|\PC(\VV)}^t)[n].
  \]

\medskip\noindent(i)
Let $A\subset\V$ be a closed, convex, subanalytic, bounded subset.
Then $\lambda_A =\{0\}$ and $\suppfct_A(-\Re\,y) = \suppfct_A(-\Re\,y) = O(|y|)$. 
Thus \eqref{eq:PWi} reads
    \begin{multline*}
      \sect_A(\V;\Db_\V) \isoto
      \{\psi\in \sect(\WW;\sho_{\WW}) \colon \\
      \exists c,\ \exists m,\ \forall y,\ |\psi(y)| \leq c(1+|y|)^m
      e^{\suppfct_A(-\Re\,y)} \}.
    \end{multline*}
(The estimates for the derivatives of $\psi$ are obtained by Cauchy formula.)
This is the classical Paley-Wiener theorem of
\cite[Theorem~7.3.1]{Hor83}.

\medskip\noindent(ii)
Let $A\subset\V$ be a closed, convex, subanalytic proper cone. 
Then $\lambda_A = A$ and $\suppfct_A = 0$. Thus  \eqref{eq:PWi} reads
    \[
    \sect_A(\V;\Db_{\V|\PR(\V)}^t) \isoto H^0\rsect(\Int
    A^\circ;\sho_{\WW|\PC(\WW)}^t).
    \]
Also this result is classical (see e.g.~\cite{Far96}).

\medskip

As another application of Theorem~\ref{thm:NHF}, consider $\VV$ and $\WW$ as dual real vector
spaces by the pairing $(x,y) \mapsto \Re\langle x,y \rangle$. Choose real coordinates
$(u) = (u',u'',u''')$ so that $\VV = \R^p\times\R^q\times\R^r$ with $p+q+r=2n$, and let
$(v) = (v',v'',v''')$ be dual real coordinates on $\WW$.

\begin{corollary}
\label{cor:qcone}
For $c\geq 0$, consider the real quadric
\[
A= \{ u^{\prime 2}-u^{\prime\prime 2} \leq c^2, \ u'''=0\} \subset \VV.
\]
Then
\[
\rsect_A(\VV;\sho_{\VV|\PC(\VV)}^t) \isoto \rsect_{\{v^{\prime
    2}-v^{\prime\prime 2} \geq 0\}}(\WW,\sho_{\WW|\PC(\WW)}^{\temp
  g})[q-n],
\] 
where
\[
g(v) = 
\begin{cases}
c \sqrt{v^{\prime 2}-v^{\prime\prime 2}}, &\text{for } v^{\prime 2}-v^{\prime\prime 2} \geq 0,\\
0,&\text{else}.
\end{cases}
\]
\end{corollary}

\begin{proof}
The proof goes as the one of Corollary~\ref{cor:PW} above using
Lemma~\ref{le:quadric} instead of Lemma~\ref{lem:conefou}.
\end{proof}

Let us describe some particular cases.

\medskip\noindent(i) Let $\R^p = \VV$ and $\R^q= \R^r=\{0\}$. Then
$A$ is a closed ball in $\VV$ centered at the origin, $g(y) = c|y|
= \suppfct_A(y)$, and we recover a particular case of \eqref{eq:PWi}.

\medskip\noindent(ii) Let $\V = \R^p\times\R^q$ and $i\V = \R^r$. Then
\[
A = {\{ (\Re x')^2-(\Re x'')^2 \leq c^2 \}} \subset \V,
\]
and we get
\[
\sect_A(\V;\Db_{\V|\PR(\V)}^t)
\isoto H^q\rsect_{\{(\Re y')^2-(\Re y'')^2 \geq
  0\}}(\WW,\sho_{\WW|\PC(\WW)}^{\temp g}).
\]
Moreover, $H^i\rsect_{\{(\Re y')^2-(\Re y'')^2 \geq 0 \}}
(\WW,\sho_{\WW|\PC(\WW)}^{\temp g})=0$ for $i\neq q$.

For $q=0$ this is the classical Paley-Wiener theorem for a closed ball in $\V$ centered at the origin.

For $c=0$ we recover a result of Faraut-Gindikin discussed in~\cite[Proposition 6.2.2]{KS97}.

\appendix

\section{Link with Tamarkin's Fourier transform}\label{se:Tam}

Categories like those in Section~\ref{se:iF} are considered by
Tamarkin in~\cite{Tam10} (see also~\cite{GS11}). 
Here, after discussing some of his constructions,
we make a connection between the
Fourier transform he considers and the functor discussed in Theorem~\ref{th:FSnh}.
We also provide a conjectural system of generators for $\R$-constructible objects in this framework.

\medskip

Let $X$ a locally compact topological space. Denote
\[
\widetilde X = X\times\R
\]
and let $t\in\R$ be the coordinate.  
For $G,G'\in\BDC(\field_{\widetilde X})$, set
\begin{equation}
  \label{eq:ttens}
  G \ttens G' = \reim s (\opb p_1 G \tens \opb p_2 G'),
\end{equation}
where the maps $p_1,p_2,s\colon X\times\R^2\to \widetilde X$ are induced
by the corresponding maps $\R^2 \to \R$ given by the first projection,
the second projection and the addition,
respectively.

Note that if $X=\V$ is a vector space, one has
\[
G\ttens\field_{\{t>0\}} \simeq G\conv\field_{\{t>0,\ x=0\}}.
\]
Thus, generalizing Notation~\ref{not:tam},
let $\tBDC(\field_{\widetilde X})$ be the full triangulated subcategory
of $\BDC(\field_{\widetilde X})$ whose objects $G$ satisfy
$G\ttens\field_{\{t>0\}} = 0$, 
or equivalently $G\ttens \field_{\{t\geq 0\}} \isoto G$.
Such categories are considered in~\cite{Tam10} and we now discuss some constructions from loc.\ cit.


\medskip

Consider the fully faithful functor
\begin{equation}
  \label{eq:etens}
  (\cdot)^\sim\colon \BDC(\field_X) \to \BDC(\field_{\widetilde X}), \quad
  F \mapsto F\etens \field_{\{t\geq 0\}}.
\end{equation}
In particular, considering the constant sheaf $\field$ on the singleton $\{pt\}$, one has
\[
\widetilde\field = \field_{\{t\geq 0\}}, \qquad \widetilde F = F
\comp[\{pt\}] \widetilde\field.
\]

Note that
\begin{equation}
  \label{eq:Kctct}
  K_{12} \comp[X_2] \widetilde K_{23} \simeq (K_{12} \comp[X_2] K_{23})^\sim.
\end{equation}

For $f\colon X_1\to X_2$, denote by $\widetilde f\colon \widetilde X_1\to \widetilde X_2$ the map
$\widetilde f = f\times \id_\R$.  
Recall the notation $X_{ij} = X_i\times X_j$.
For $L_{ij}\in\BDC(\field_{\widetilde X_{ij}})$ set
\begin{equation}
  \label{eq:tcirc}
  L_{12} \tcomp[X_2] L_{23} = \reim{\widetilde q_{13}}(\opb{\widetilde
    q_{12}}L_{12} \ttens \opb{\widetilde q_{23}}L_{23}).
\end{equation}

Note that $G \ttens G' \simeq G \tcomp[X] G'$
and that one has
\begin{equation}
  \label{eq:kkisk}
  \widetilde\field \tcomp[\{pt\}] \widetilde\field \simeq \widetilde\field.
\end{equation}

\begin{proposition}
  \label{pr:compot}
  For $K_{12}\in\BDC(\field_{X_{12}})$ and
  $L_{23}\in\BDC(\field_{\widetilde {X_{23}}})$ one has
  \[
  \widetilde K_{12} \tcomp[X_2] L_{23} \simeq K_{12} \comp[X_2]
  (L_{23} \tcomp[\{pt\}] \widetilde \field) \quad\text{in
  }\BDC(\field_{\widetilde {X_{13}}}).
  \]
\end{proposition}

\begin{proof}
  For $y = (x_1,x_2,x_3,t,t') \in X_1 \times X_2 \times X_3 \times
  \R^2$, set $p(y) = (x_1,x_2)$, $q(y) = (x_2,x_3,t')$, $r(y) = t$ and
  $u(y) = (x_1,x_3,t+t')$. Then both sides are isomorphic to $\reim u
  (\opb p K_{12} \tens \opb q L_{23} \tens \opb r \field_{\{t\geq
    0\}})$.
\end{proof}

\begin{corollary}
  For $K_{ij}\in\BDC(\field_{X_{ij}})$ one has
  \[
  \widetilde K_{12} \tcomp[X_2] \widetilde K_{23} \simeq (K_{12}
  \comp[X_2] K_{23})^\sim \quad\text{in }\BDC(\field_{\widetilde {X_{13}}}).
  \]
\end{corollary}

\begin{proof}
  One has
  \begin{align*}
    \widetilde K_{12} \tcomp[X_2] \widetilde K_{23}
    &\simeq K_{12} \comp[X_2] (\widetilde K_{23} \tcomp[\{pt\}] \widetilde\field) \\
    &\simeq K_{12} \comp[X_2] (K_{23} \comp[\{pt\}] (\widetilde\field \tcomp[\{pt\}] \widetilde\field)) \\
    &\simeq K_{12} \comp[X_2] (K_{23} \comp[\{pt\}] \widetilde\field) = K_{12} \comp[X_2] \widetilde K_{23} \\
    &\simeq (K_{12} \comp[X_2] K_{23})^\sim.
  \end{align*}
  Where the first isomorphism follows from
  Proposition~\ref{pr:compot}, the third isomorphism follows from
  \eqref{eq:kkisk} and the last isomorphism from \eqref{eq:Kctct}.
\end{proof}

Note that \eqref{eq:tcirc} induces a functor
\[
  {\tcomp[X_2]}\colon \tBDC(\field_{\widetilde X_{12}}) \times
  \tBDC(\field_{\widetilde X_{23}}) \to \tBDC(\field_{\widetilde X_{13}}).
\]

\medskip

Let $\V$ and $\W$ be dual real vector spaces by
the pairing $(x,y)\mapsto\langle x,y \rangle$.
In~\cite{Tam10}, the following analogue of the
Fourier-Sato transform is considered:
\[
\Phi\colon\tBDC(\field_\Vt) \to \tBDC(\field_{\Wt}),\qquad
G \mapsto G \tcomp[\V] \field_{\{\langle x,y\rangle \leq t\}}.
\]
Note that
\begin{equation}
\label{eq:xyt}
  \field_{\{\langle x,y\rangle \leq t\}} \tcomp[\{pt\}] \widetilde
  \field \simeq \field_{\{\langle x,y\rangle \leq t\}}.
\end{equation}

Recall the notations of section~\ref{se:iF}.

\begin{proposition}
  For $F\in \BDC(\field_\V)$ one has
  \[
  \Phi(\widetilde F) \simeq (\reim i F)^\wedge.
  \]
\end{proposition}

\begin{proof}
  By Proposition~\ref{pr:compot} and \eqref{eq:xyt}, one has
  \[
  \Phi(\widetilde F) = \widetilde F \tcomp[\V] \field_{\{\langle
    x,y\rangle \leq t\}} \simeq F \comp[\V] (\field_{\{\langle
    x,y\rangle \leq t\}}\tcomp[\{pt\}] \widetilde\field)
\simeq F \comp[\V] \field_{\{\langle
    x,y\rangle \leq t\}}.
  \]
The statement then follows from Lemma~\ref{lem:FiF}.
\end{proof}

\medskip

Let now $j\colon X\to X'$ be an open subanalytic embedding of real analytic manifolds.
Let us still denote by $j$ the embedding of $\widetilde X$ in $\widetilde X' = X'\times\PR(\R)$.

\begin{conjecture}
\label{con:Tam}
Any $G\in\BDC_{j\text-\Rc,\,*\{t\geq 0\}}(\field_{\widetilde X})$
is isomorphic to a bounded complex $G^\bullet$ where each $G^i$ is a
direct sum, locally finite in $\widetilde X'$, of sheaves of the form
$\field_{\{x\in U,\,t\geq f(x)\}}$ for $U\subset X$ an open
$j$-subanalytic subset
and $f\colon U\to\R$ a continuous $j$-subanalytic function. 
\end{conjecture}

	


\providecommand{\bysame}{\leavevmode\hbox to3em{\hrulefill}\thinspace}

\end{document}